\documentclass[11pt]{article}

\usepackage[margin=2.12cm]{geometry}
\usepackage[caption=false]{subfig}
\usepackage{amsmath,amsthm,amsfonts,amssymb}
\usepackage{graphicx}
\usepackage{xcolor}
\usepackage{dsfont}
\usepackage{microtype}
\usepackage{placeins}
\DisableLigatures[f]{encoding = *}
\usepackage{ellipsis} 
\usepackage{bm}
\usepackage[titletoc,title]{appendix}

\newcommand{\E}{\mathbb{E}}
\newcommand{\tr}{\mathrm{tr}}
\newcommand{\R}{\mathbb{R}}
\renewcommand{\a}{\bm{\alpha}}
\newcommand{\Sa}{C_{\bm{\alpha}}}
\newcommand{\N}{\mathbb{N}}
\renewcommand{\l}{\lambda}
\renewcommand{\b}{\beta}

\newcommand{\Ak}{\mathcal{B}_k}
\newcommand{\g}{\gamma}
\newcommand{\var}{\mathrm{Var}}
\newcommand{\cH}{\mathcal{H}}

\newcommand{\pp}{\mathbf{p}}
\newcommand{\rt}{R_{\mathrm{tot}}}
\newcommand{\bmu}{\bm{\mu}}
\newcommand{\bmo}{\bm{1}}
\newcommand{\bS}{\Sigma}
\newcommand{\bJ}{\bmo \bmo^T}
\renewcommand{\o}{\bm{\omega}}

\newtheorem{theorem}{Theorem}

\newtheorem{lemma}[theorem]{Lemma}
\newtheorem{proposition}[theorem]{Proposition}

\begin{document}

\title{Optimization of stochastic lossy transport networks\\and applications to power grids}

\author{Alessandro Zocca\thanks{Department of Mathematics, Vrije Universiteit Amsterdam, 1081 HV, Amsterdam, The Netherlands, Email: \texttt{a.zocca@vu.nl}} \, and Bert Zwart\thanks{Centrum Wiskunde and Informatica (CWI), 1098 XG Amsterdam, The Netherlands}}

\maketitle

\begin{abstract}
Motivated by developments in renewable energy and smart grids, we formulate a stylized mathematical model of a transport network with stochastic load fluctuations. Using an affine control rule, we explore the trade-off between the number of controllable resources in a lossy transport network and the performance gain they yield in terms of expected power losses. Our results are explicit and reveal the interaction between the level of flexibility, the intrinsic load uncertainty and the network structure.
\end{abstract}


\section{Introduction}
\label{sec1}
A transport network is an abstract model describing a structure in which some commodity is transferred from the ``source'' nodes of the network to the ``sink'' nodes according to a specified routing that is determined by some external principle or design, see~\cite{Daganzo1997, Whittle2007}. Examples of transport networks are road networks, railways, pipes, and power grids.

In this work we focus in particular on \textit{lossy} transport networks where a fraction of the transported good is inevitably lost, having in mind as primary application power systems in which part of the transported electricity is lost due to heat dissipation in the transmission lines.

The main question that we want to address in the present paper is how these transportation losses can be reduced in the scenario in which we have no direct control on the routing, but some of the nodes of the network have controllable loads. This is the case for power systems in which the line flows are determined by physical laws, but at the same time feature an increasing number of controllable energy resources, like energy storage devices, smart buildings and appliances, and electric vehicles.

In the present work we consider a probabilistic model to describe the stochastic fluctuations of the load in (a subset of) the nodes of the network. This is instrumental to model the stochastic load fluctuations due to power demand uncertainty and intermittent generation by renewable energy sources. The current power grids were originally built around conventional power generation systems and therefore they are not equipped to cope with this massive amount of uncertainty, especially in power supply. Managing this uncertainty on such a large scale with existing methods will soon become crucial: in the next decades power grids will have to become more flexible and robust to reduce the likelihood of contingencies and blackouts, whose social and economic impact is enormous.

As mentioned earlier, next to the increasing renewable penetration, there is another powerful trend that is driving this pervasive evolution of power system: the advent of distributed energy resources. At a high level, all these resources can be seen as ``virtual storage/batteries'', in the sense that they can dynamically reduce their power consumptions and even inject electricity in the power grid when necessary, see~\cite{BusicMeyn2017,Taylor2015,Barooah2019}. Even if at the present stage these resources are not fully incorporated, they have a huge potential: if their penetration increases and we can actively and optimally control them, then they can make power grids more flexible and at the same time effectively mitigate the volatile nature of renewable power generation and allow a higher share of renewable energy sources.

The controllable loads that we consider in this work should be seen as an abstraction of more concrete examples, such as (i) actual energy storage that is neither full nor empty, (ii) distribution grids with ample flexible and/or deferrable load \cite{Mathias2016b,BusicMeyn2017,Meyn2013,Meyn2015} or (iii) conventional generators that can provide balancing services.

The stochastic network model considered in this paper aims to understand the potential that these controllable resources can have in mitigating the load uncertainty and in particular the transportation losses due to the stochastic load fluctuations in the network nodes. Specifically, we consider a network with random sources and sinks modeled by an undirected connected weighted graph $G$ consisting of $n$ nodes and $m$ edges to which are associated non-negative weights $\bm{\b} \in \R_+^{m}$, and investigate how much could the \textit{average} total loss be reduced by operating optimally the subset $B \subseteq V$ of nodes with controllable loads.

The metric we consider here to quantify the transportation losses due to stochastic fluctuations is a quadratic function of the load profile that has been introduced in~\cite{JohnsonChertkov2010} as proxy for the total power losses in AC power grids, as we will review in more detail in the next section. 
Such a metric generalizes the notion of \textit{total effective resistance} $\rt(G)$ of the graph $G$, also known as \textit{Kirchhoff index}. This is a key quantity that measures how well connected and robust a network is and for this reason has been extensively studied and rediscovered in various contexts, such as complex network analysis~\cite{Ellens2011} and theoretical chemistry (for an overview see~\cite{Zhou2008} and references therein).

In this paper, we address the question of how to optimally operate distributed energy resources to minimize power losses in power systems. Due to imperfect sensing and delayed communication, a real-time perfect coordination between network nodes based on the realized fluctuations is unfeasible. We consider thus a static decentralized policy that amounts to an affine control rule for the controllable loads, which is inspired by the Automatic Generation Control (AGC) mechanism, currently used in power systems, cf.~\cite{Wood2014}. This control mechanisms prescribes that in case of a power imbalance, all the generators have to either increase or decrease their power generation \textit{proportionally to their partecipation factors}, which are static nonnegative scalars set beforehand based on economical principles.

In the present work, we envision a scenario where not just a few big conventional generators, but several other distributed energy resources could also join this effort to balance power fluctuations. We thus formulate a constrained optimization problem to find the optimal static \textit{load-sharing factors} for the controllable loads. We prove that the solution is unique and give a closed-form expression for such an optimal control from which reflects the interplay between the network structure, the location of the controllable nodes, and the correlation structure of the load fluctuations.

In particular, our result shows how the correlation structure between load fluctuations affects the optimal operations of the network resources and this is extremely relevant in power grids with geographically close wind or solar farms, whose power outputs are obviously highly correlated. These insights are derived without making Gaussian assumptions on the distribution of the fluctuations. 


We then use this explicit solution to explore the trade-off between the number of controllable resources available in a network and the performance gain they yield, quantified as transportation loss reduction. The analysis builds on and extends that of~\cite{ZZ16}, which considered the scenario with only two controllable loads. In particular, the authors show therein that the expected total loss due to fluctuations can be reduced by $25\%$ in a line network by adding one controllable storage device. In the present paper we extended that insight by showing that for large graphs the total expected loss can be reduced on average by a factor $(1+1/k)/2$ in the scenario where $k$ controllable loads are available. The key insight is larger values of $k$ yield diminishing improvements in terms of total expected loss.

This suggests that, even if power grids are becoming locally more robust, transportation losses can be reduced by up to 50\%, and the number of controllable loads or balancing services $k$ quantify how close one can get to this reduction. Though our model is stylized, it provides a simple quantitative estimate on the value of balancing services in a scenario where each node in the network is self-sufficient on average.



We remark that our stylized model is ``static'', in the sense the optimal control we derive does rebalance the total power mismatch in the network in any scenario, but does not depend on the \textit{realized} load fluctuations. Indeed, it depends only on the network structure, on the location of the controllable loads and on their average covariance structure of the load fluctuations. For this reason our model is not specific for a precise time-scale and provides insights into the value of balancing services both in real-time operations as well as long-term planning.

Our work provides a new mathematical approach which can be of help for the design of future power grids and of control schemes for distributed energy resources. In this respect, it complements a large body of literature on optimal topology design for power grids~\cite{Deka2017,JohnsonChertkov2010,GBS08} and on optimal control of multiple controllable devices and/or generators where often the designed participation factors are also affine in the stochastic load fluctuations, see e.g.,~\cite{Apostolopoulou2014,Bienstock2016,Bienstock2017,Bienstock2016a,Chertkov2017d,Guggilam2017b,Guggilam2017c,LubinDvorkinBackhaus2016,Kanoria2011,RoaldChertkov2016,Roald2017a,Sjodin2012}.
Optimal policies for storage management, especially aiming at the mitigation of the uncertainties in wind generation, have been explored in~\cite{BGK12,Gast2014b,Gast2012,VanDeVen2013}, where, however, the physical network is not modeled explicitly. Optimal storage placement can increase network reliability, as shown in~\cite{BCZ16} by simulation techniques. Storage can also be used for arbitrage, exploiting temporal price differences~\cite{CGZ14,CZ15} and the impact of storage on energy markets has been studied in~\cite{Cruise2016,Gast2013b}.

The rest of the paper is organized as follows. We provide a detailed model description in Section \ref{sec2}. In Section \ref{sec3} we investigate the optimal load sharing factors in several scenarios. These results are applied in Section \ref{sec4} where a scaling law is presented for large networks. A more general optimization problem, which takes into account economic factors and/or further limitations of the controllable loads is presented in Section~\ref{newsec}. In Section \ref{sec5} we report several numerical experiments and, lastly, Section \ref{sec6} concludes.

\section{Model description}
\label{sec2}
In this paper we model a lossy transport network as a weighted graph $(G,w)$, where the graph $G$ is a simple undirected graph $G=(V,E)$ with $|V| =n$ nodes labelled as $V=\{1,\dots,n\}$, and $|E|=m$ edges and $\bm{\beta} \in \R_+^m$ is the collection of edge weights. In the context of power grids, the nodes of $G$ are often referred to as \textit{buses}, the edges as \textit{transmission lines} and the quantity $\b_{i,j}$ is the \textit{susceptance} of the transmission line $\ell=(v,w) \in E$ connecting buses $i$ and $j$. Missing edges can be thought as edges with zero weight.

The \textit{weighted Laplacian matrix} of the graph $G$, often referred to also as the \textit{susceptance matrix} of $G$ in the case of a power system, is the matrix $L\in \R^{n\times n}$ defined for every $i,j =1,\dots,n$ as
\[
    L_{i,j} :=
    \begin{cases}
    	\sum_{k \neq i} \b_{i,k}	& \text{ if } i=j,\\
		- \b_{i,j}					& \text{ if } i \neq j.
	\end{cases}
\]
It is well-known that $L$ is a real symmetric positive semi-definite matrix. By construction, all the rows of $L$ sum up to zero and thus the matrix $L$ is singular. Under the assumption that $G$ is a connected graph, the eigenvalue zero has multiplicity one and the corresponding eigenvector is equal to the vector with all unit entries, which we denote by $\bm{1} \in \R^n$. 

Let $L^+ \in \R^{n \times n}$ be the \textit{Moore-Penrose pseudoinverse} of the weighted Laplacian matrix $L$. Using the eigenspace structure of $L$, the pseudoinverse $L^+$ can be defined as
\[
	L^+ :=\left  (L+ \frac{\bm{1}\bm{1}^T}{n}\right )^{-1} -\frac{\bm{1}\bm{1}^T}{n}.
\]
The matrix $L^+$ is also real, symmetric, and positive semi-definite. For the proof of these properties of the matrices $L$ and $L^+$ and for further spectral properties of graphs, we refer the reader to~\cite{RanjanZhangBoley2014,VanMieghem2011}. 

Denote by $\pp \in \R^{n}$ the \textit{load profile} at the network nodes, where $\pp_i$ is the load at node $i$  for every $i=1,\dots,n$. In the context of power grids the $i$-th entry of the vector $\pp$ models the power generated (if $\pp_i>0$) or consumed (if $\pp_i<0$) at node $i$. We say that a load profile $\pp \in \R^n$ is \textit{balanced} if $\mathbf{1}^T \pp=0$.

Given a network with a balanced load profile $\pp \in \R^n$, we define its \textit{total loss} $\cH=\cH(\pp)$ as
\[
 	\cH:= \frac{1}{2} \pp^T L^+ \pp.
\]
The scalar quantity $\cH$ is a quadratic form of the load profile vector $\pp$ and, as such, is always non-negative, thanks to the fact that $L^+$ is a positive semi-definite matrix.

The total aggregated loss $\cH$ will  be central in our analysis, being the most natural choice for an efficiency metric in a lossy transport networks such as power systems. 
More specifically, $\cH$ has been shown by~\cite{JohnsonChertkov2010} to be good approximation for \textit{total power loss} in AC power grids, as we will briefly outline. 

AC current flows are described in terms of complex amplitudes and lines with complex impedances. In a power transmission networks operating in stationary conditions (i) all voltages are mostly constants and (ii) we can ignore reactive power flows and line conductances since they are an order of magnitude smaller than the active power flows and line susceptances, respectively. Under these two assumptions, it is reasonable to consider the linearized version of the AC Kirchhoff equations for the power flows, i.e., the so-called \textit{DC-approximation}, called in this way since it resembles in structure the equations describing a resistive network with DC flows, cf.~\cite[Chapter 6]{Wood2014}. If $\pp$ denotes the vector of active power injections in the network nodes, by keeping only the leading term of such a DC-approximation, the power flows $\bm{f} \in \R^m$ on the network lines are given by
\begin{equation}
\label{eq:fLp}
	\bm{f} = \Lambda L^+ \pp,
\end{equation}
where $L$ is the weighted Laplacian matrix derived looking only the imaginary part of the network admittance matrix and $\Lambda \in \R^{m \times n}$ is the corresponding weighted edge-vertex incidence matrix
\[
	\Lambda_{\ell,i} =
	\begin{cases}
	\beta_\ell & \text{ if } \ell=(i,j),\\
	-\beta_\ell & \text{ if } \ell=(j,i),\\
	0 & \text{ otherwise.}
	\end{cases}
\]
Since the power loss on each line $\ell$ is proportional to $\beta_\ell^{-1} \bm{f}_\ell^2$, the aggregated power loss is (up a constant factor capturing the conductance-to-admittance ratio of the lines) equal to
\[
	\cH = \frac{1}{2} \sum_{\ell \in E} \beta_\ell^{-1} \bm{f}_\ell^2 = \frac{1}{2}  \bm{f}^T \mathrm{diag}(\beta_1,\dots,\beta_m)^{-1} \bm{f} = \frac{1}{2}  \pp^T L^+ \Lambda^T \mathrm{diag}(\beta_1,\dots,\beta_m) \Lambda L^+ \pp = \frac{1}{2} \pp^T L^+ \pp,
\]
where we used the fact that $L=\Lambda^T \mathrm{diag}(\beta_1,\dots,\beta_m)^{-1} \Lambda$ and that $L^+ L L^+ = L^+$. Therefore, modulo a trivial rescaling, the quantity $\cH$ is an approximation for the total power losses calculated using only the leading order DC-approximation of the AC network flows.

The quantity $\cH$ has also been considered as a scalar measure of the \textit{network tension} in~\cite{GuoLiangLow2017,LaiLow2013}, where it is shown to be monotone along a cascade failure, as long as the network remains connected. Lastly, as we will show later, $\cH$ can also be seen as a generalized total effective resistance in the sense that it quantifies how robust the network $G$ is against a stochastic load profile with a predefined covariance structure.

\subsection{Stochastic fluctuations and load-sharing factors}
\label{sub21}
In this work we are particularly interested in a transport network with a \textit{stochastic} load profile, which means that we take $\pp$ to be a multivariate random variable.

More precisely, we assume $\pp$ to be of the form $\pp = \bmu + \o$, where $\bmu\in \R^n$ is the nominal load profile in the network and $\o$ is a $n$-dimensional random vector modeling the fluctuations. We henceforth assume that $\E \o = \bm{0}$ and that $\o$ has finite second moment. We further denote by $\bS \in \R^{n\times n}$ the covariance matrix of the load fluctuations, namely $\bS_{i,j}=\mathrm{cov}(\o_i,\o_j)=\E[ \o_i \o_j] < \infty$.
Nodes in which there are no stochastic load fluctuations can be modeled by setting all the entries equal to zero in the corresponding row and column of the matrix $\bS$. We denote by $S \subseteq V$ the subset of nodes with stochastic load fluctuations and we will henceforth assume that $|S|\geq 1$.

We further assume that the nominal load profile is balanced on average, namely $\bm{1}^T\bmu=0$; this assumption is reasonable as every 5-15 minutes the setpoint $\bmu$ is adjusted by solving the so-called Optimal Power Flow (OPF) problem which specifically constraint the net total power to be equal to zero. This assumption, however, does not guarantee that every stochhastic realization of the load profile is balanced: indeed, the total mismatch $\bm{1}^T \pp$ in the network is a random variable, which can be expressed as the net sum of fluctuations, namely
\[
	\bm{1}^T \pp = \bm{1}^T(\pp - \bmu) = \bm{1}^T \o = \sum_{i=1}^n \o_i.
\]
Let $\sigma^2:=\var \Big (\sum_{i=1}^n \o_i \Big)$ be the variance of the sum of the load fluctuations, which also rewrites as
\begin{equation}
\label{eq:trSJ}
	\sigma^2 =  \sum_{i,j=1}^n \bS_{i,j} = \bmo^T \bS \bmo = \tr(\bS \bmo \bmo^T).
\end{equation}
Since we assumed that there is at least one node with stochastic load fluctuations, we have $\sigma^2>0$.

In order to cope with the stochastic fluctuations and, in particular, to keep the network balanced, we assume that load at each node is \textit{controllable}: for every $i=1,\dots,n$, node $i$ can deal with (either generate or store) a controllable fraction $\a_i \in \R$ of the realized total mismatch $ \bm{1}^T(\pp - \bmu) = \sum_{i=1}^n \o_i$. In other words, we assume that while using the \textit{load-sharing factors} $\a = (\a_1,\dots,\a_n) \in \R^n$ the \textit{net} load profile $\pp(\a)$ is given by
\[
	\pp(\a) =  \Big (\pp_1 - \a_1 \sum_{i=1}^n \o_i, \, \dots \, , \pp_n - \a_n \sum_{i=1}^n \o_i\Big ) =  \Big (\bmu_1 +\o_1 - \a_1 \sum_{i=1}^n \o_i, \, \dots \, , \bmu_n +\o_n - \a_n \sum_{i=1}^n \o_i\Big ).
\]
For any $j=1,\dots,n$, the term  $\a_j \sum_{i=1}^n \o_i$ corresponds to the power generated or stored in the corresponding controllable load when using an affine control responsive to stochastic load fluctuations.
We can rewrite the net power load profile $\pp(\a)$ in vector form as
\begin{equation}
\label{eq:Sap}
	\pp(\a) = \Sa \pp = \Sa (\bmu+ \o) = \bmu + \Sa \o,
\end{equation}
where $\Sa \in \R^{n \times n}$ is the matrix defined as
\[
	\Sa:=I - \a \, \bm{1}^T = \left(
	\begin{array}{ccccc}
	1-\a_1 & -\a_1 & \dots &  & -\a_1 \\
	-\a_2 & 1-\a_2 & -\a_2 & \dots & -\a_2 \\
	 & \ddots & \ddots & \ddots &  \\
	\vdots &  &  & & \vdots \\
	-\a_n &  & \dots & -\a_n & 1-\a_n
	\end{array} \right ).
\]
The last equality in~\eqref{eq:Sap} follows from the fact that $\bm{1}^T\bmu=0$, since $\Sa \bmu =(I - \a \, \bm{1}^T) \bmu = \bmu - 0 \cdot \a = \bmu$. When the load-sharing factors $\a$ are used, the total mismatch in the network is equal to
\[
	\bm{1}^T \pp(\a) = \bm{1}^T \Sa (\bmu+\o) =\bm{1}^T (I - \a \, \bm{1}^T ) \o = (\bm{1}^T - (\bm{1}^T \a) \bm{1}^T) \o.
\]
From this expression it is immediate to derive the condition on $\a$ that guarantees that the net load profile $\pp(\a)$ is balanced for \textit{any} realization $\o$, which is stated in the next lemma.

\begin{lemma}[Stochastic load profile balance condition]%
\label{lem:asum}
Consider a network with balanced nominal load profile $\bm{1}^T \bmu = 0$ and controllable loads using load-sharing factors $\a \in \R^n$. The stochastic load profile $p(\a)$ is balanced for every realization of the fluctuations if and only if 
\begin{equation}
\label{eq:asum}
	\bm{1}^T \a = \sum_{i=1}^n \a_i = 1.
\end{equation}
\end{lemma}

\subsubsection*{Model discussion}
Before presenting our results in the next section, we briefly discuss here the motivation behind the assumptions on the distribution of $\o$ and the motivation for considering an affine static control policy for the fluctuations.

First of all, the random variable $\sum_{j=1}^n \o_j$, which describes the net power mismatch, can have any sign. We thus tacitly assume that our controllable nodes can respond to both positive and negative fluctuations, meaning that they can both store energy in excess or provide energy if the network needs it. For sake of generality, we impose no further constraints on the random variable $\sum_{j=1}^n \o_j$, and thus the power absorbed/injected by controllable node $i$, i.e., $\a_i \sum_{j=1}^n \o_j$, could theoretically take very large (positive and negative) values. This may seem a dubious and unrealistic assumption, but it is not for various reasons, which we now outline.
Firstly, in normal operating conditions and on short time scales the quantity $\sum_{j=1}^n \o_j$ is typically ``small''. Indeed, our model is particularly relevant for a prompt response on short-time scales, in which \textit{large unforeseen power fluctuations are highly unlikely}. If an extremely large net power fluctuation occurs, it would be physically impossible for controllable loads and storage devices to resolve this imbalance by themselves and the network operator needs to take specific ad-hoc emergency actions anyway. In particular, the real-time energy market and a consequent Optimal Power Flow routine on the 5-15 minutes scale will allow for a new safe setpoint $\bmu$ and for energy reserves to be used/bought.
Secondly, since a detailed statistical modeling of demand fluctuations or renewable power production is not within the scope of this paper, we consider a general multi-dimensional distribution for the stochastic fluctuations $\o_1,\dots,\o_n$. It is however very reasonable to assume that the support of these random variables is bounded, since, for instance, renewable energy sources have cannot produce a negative amount of power nor produce beyond a specific power rating (above which automatic protective relay mechanisms are automatically triggered causing to so-called power curtailment).
Lastly, a controllable load that \textit{at a specific moment} does not have the two-directional flexibility (i.e., that cannot both store and release energy) can temporarily ``go offline'' and we can account for it by adding the constraint $\a_i=0$ for the corresponding network node.

An affine control modeled using the load-sharing factors $\a \in \R^n$ is a clearly simplification, especially when the controllable loads models energy storage. Indeed, it does not incorporate many details, among which possible ramping constraints or the current state of charge. In our model such details are omitted on purpose to have a mathematically tractable optimization problem and to better identify the interplay between load uncertainty and storage operations. A further simplification we make is that we allow to choose $\a$ without invoking line limits, as we also do this for mathematical tractability, but the extension we present in Section~\ref{newsec} accounts for these limits and several others.

We remark that these assumptions on the stochastic fluctuations $\o$ and on the affine control of the deviations from the nominal setpoints are fairly common in power systems operations, see e.g.~\cite{BCH14,RoaldChertkov2016,Roald2017a}. 



\subsection{Expected total loss: definition and properties}
For any $1\leq k \leq n$, we can model a network where exactly $k$ nodes have controllable loads by imposing that the remaining $n-k$ nodes have load-sharing factors equal to zero, so that for any such node $\pp_i=\pp_i(\a)$. Using the net load profile $\pp(\a)$, the total loss rewrites as
\[
	\cH(\a) = \frac{1}{2} \pp(\a)^T L^+ \pp(\a),
\]
and therefore $\{\cH(\a)\}_{\a \in \R^n}$ is a family of random variables parametrized by the control $\a$. Being a quadratic form and being $L^+$ a positive semi-definite matrix, it immediately follows that $\cH(\a)$ is a non-negative random variable for any $\a\in \R^n$.

The next proposition, which is proved in Appendix~\ref{app:a}, shows how, leveraging the properties of the matrices $L^+$ and $\Sa$, the expected total loss $\E \cH(\a)$ rewrites as the sum of two contributions, one stochastic and one deterministic which is not affected by the control $\a$. Furthermore, it rewrites the expected total loss as a quadratic function the load-sharing vector $\a$.

\begin{proposition}
\label{prop:Eha}
Consider a network with a balanced nominal load profile $\bmu$ and zero-mean stochastic fluctuations $\o$ with covariance matrix $\bS$. Then, the expected total loss using the control $\a$ is given by
\begin{equation}
\label{eq:EH}
	\E \cH(\a) = \E \cH_s(\a) + \frac{1}{2} \bmu^T L^+ \bmu,
\end{equation}
where $\E \cH_s(\a)$ is the expected total loss due to the stochastic fluctuations. Furthermore, $\E \cH_s(\a) \geq 0$ for every $\a \in \R^n$, and the following identity holds:
\begin{equation}
\label{eq:EHs}
	\E \cH_s(\a) = \frac{\sigma^2 }{2} (\a^T L^+ \a) -  \bm{1}^T \bS L^+ \a  + \frac{1}{2} \tr(\bS L^+).
\end{equation}
\end{proposition}
The first important remark is that the expected total loss $\E \cH(\a)$ is a quadratic form in the vector $\a \in \R^n$ since it can be rewritten as
\begin{equation}
\label{eq:quadform}
 	\E \cH(\a) = \frac{\sigma^2}{2} \a^T A \a -  \bm{b}^T \a  + c,
\end{equation}
where $A= L^+$ is a positive semi-definite  matrix, $\sigma^2 >0$, $\bm{b}= L^+ \bS \bm{1} \in \R^n$, and $c = \tr(\bS L^+)/2 + \frac{1}{2} \bmu^T L^+ \bmu \in \R^+$. Furthermore, we can already conclude that the nominal load profile $\bmu$ has no impact on the optimal control $\a$, since it appears only in the constant term $c$.


Before investigating what is the optimal load-sharing vector for a given network and covariance structure of the noise, we argue here why the quantity $\E \cH(\a)$ can be seen as a generalized notion of effective resistance. In order to do so, we will first recall some classical definitions.

The \textit{effective resistance} $R_{i,j}$ between a pair of nodes $i$ and $j$ of the network $G$ is defined as the electrical resistance measured across nodes $i$ and $j$ when we look at $G$ as electrical network in which resistors with conductance $\b_1^{-1},\dots,\b_m^{-1}$ are placed at the corresponding network edges. Equivalently,
\[
	R_{i,j} := (\mathbf{e}_i - \mathbf{e}_j)^T L^+ (\mathbf{e}_i - \mathbf{e}_j) = L^+_{i,i} + L^+_{j,j} - 2 L^+_{i,j},
\]
where $\mathbf{e}_i$ denotes the vector with a $1$ in the $i$-th coordinate and $0$'s elsewhere. The \textit{total effective resistance} of a graph $G$ is then defined as
\[
	\rt(G):=\frac{1}{2} \sum_{i,j=1}^n R_{i,j}.
\]
The same quantity is also known as \textit{Kirchhoff index} when the network $G$ is such that all the edge weights are equal to $1$. As mentioned in the introduction, the total effective resistance has been used in various contexts~\cite{Ellens2011,Zhou2008} to quantify how well connected a given network is. 
In the context of electrical networks the total effective resistance $\rt(G)$ is shown in~\cite{GBS08} to be proportional to the average power dissipated in a resistive DC network $(G,w)$ when random i.i.d.~currents with zero mean and unit variance are injected at the nodes.

We can look at the network $(G,w,\a)$ with controllable loads introduced earlier as a flexible transport network, where the load-sharing factors $\a_1,\dots,\a_n$ can be tuned to respond optimally to specific stochastic load fluctuations. In this respect, we claim that the quantity $\E \cH_s(\a)$ can be seen as a (rescaled) \textit{generalized} total effective resistance that measures how ``robust'' the network $(G,w,\a)$ is against stochastic load fluctuations with covariance structure $\bS$. To further corroborate this claim, we now show that the expected total loss reduces to the classical total effective resistance $\rt(G)$ in the special case where the load-sharing factors are all equal, i.e.,~$\a_i=1/n$ for every $i=1,\dots,n$, and the stochastic load fluctuations are i.i.d.~random variables with with zero mean and unit variance, i.e.,~$\bS=I$. Using~\eqref{eq:EHs}, the expected total loss due to fluctuations rewrites as
\begin{equation}
\label{eq:cHrt}
	 \E \cH_s\Big(\frac{1}{n} \bm{1} \Big) = \frac{n}{2 n^2} (\bm{1}^T L^+ \bm{1}) -  \bm{1}^T L^+ \bm{1}  + \frac{1}{2} \tr(L^+) = \frac{1}{2} \tr (L^+) = \frac{\rt(G)}{2n},
\end{equation}
where in the last step we use the well-known identity $\rt(G) = n \cdot \tr(L^+)$ proved by~\cite{KR93} that relates the total effective resistance of a graph with its spectrum.

Rayleigh's monotonicity principle~\cite{Doyle2000} states that the pairwise effective resistance $R_{i,j}$ is a non-increasing function of the edge weights and, as a consequence, also the total effective resistance $\rt(G)$ is. The following proposition, proved in Appendix~\ref{app:a}, shows that a similar property also holds for $\cH(\a)$: regardless of the load-sharing vector $\a$, the total power loss does not increase when edges are added or weights are increased.

\begin{proposition}[Monotonicity of total power loss] \label{prop:monotone}
Let $G$ be a weighted connected graph and let $G'$ the graph obtained from $G$ by increasing the weight of edge $e=(i,j)$ by $\beta >0$ or by adding the edge $e=(i, j)$ with weight $\beta >0$. For any load-sharing vector $\a \in \R^n$ and any realization of the stochastic load fluctuations $\o$, the following inequality holds
\[
	\cH^{G'} (\a) \leq \cH^{G}(\a),
\]
and, in particular, $\E \cH^{G'} (\a) \leq \E \cH^{G}(\a)$.
\end{proposition}



\section{Optimal load-sharing control}
\label{sec3}
In this section we consider the problem of minimizing the expected total loss $\E \cH(\a)$ given a network $G$ and a stochastic load covariance structure $\bS$.

Let $B \subseteq V$ be the subset of $k$ nodes with controllable loads. We focus first on the scenario in which not all the nodes have controllable loads and thus assume $1\leq k < n$. Besides the constraint~\eqref{eq:asum}, we further add $n-k$ constraints on the optimal load-sharing vector $\a \in \R^n$ to account for the absence of controllable loads in the nodes in $V \setminus B$, obtaining the following constrained optimization problem in $\R^n$:%
\begin{align}
		\min_{\a \in \R^n} \quad & \E \cH_s(\a) \nonumber\\
		\text{s.t.} \quad & \bm{1}^T \a = 1, \label{eq:optk}\\
		& \a_{v} =0, \quad \forall\, v \in V \setminus B. \nonumber
\end{align}
Note we consider only the expected total loss due to stochastic load fluctuations in the objective function, since in view of~\eqref{eq:EH} it differs only by a constant from the expected total loss.

We henceforth assume that the $k$ nodes with the controllable loads are those with labels $1,\dots,k$, i.e.,~$B=\{1,\dots,k\}$. We can make this assumption without loss of generality as it amounts to a relabelling the network nodes. If this is the case, the rows and columns of matrices $L$, $L^+$, and $\bS$ and the entries of the vector $\bmu$ are also rearranged accordingly.

Let $P_B \in \{0,1\}^{n \times k}$ be the binary matrix that maps any $k$-dimensional vector $\widetilde{\a} \in \R^k$ to the $n$-dimensional vector $P_B \widetilde{\a} = (\widetilde{\a}_1,\dots,\widetilde{\a}_k, 0, \dots, 0) \in \R^n$. Such a matrix can be defined component-wise as $(P_B)_{i,j} := \delta_{\{ i = j \}} \delta_{\{ i \leq k\}}$, for $i=1,\dots,n$ and $j=1,\dots,k$, and has the following structure:
\[
	P_B =
	\begin{pmatrix}
		\displaystyle I_k\\
		\displaystyle \mathbb{O} \\
	\end{pmatrix},
\]
where $I_k$ is the $k\times k$ identity matrix and $\mathbb{O} \in \R^{n-k \times k}$ is a matrix with all entries equal to zero.

In our first main result we present a closed-form expression for the optimal load-sharing factors of $k$ controllables.
\begin{theorem}[Optimal load-sharing between $k<n$ controllables] \label{thm:optk}
Consider a network with balanced nominal load profile $\bmo^T \bmu=0$ in which the nodes in $B=\{1,\dots,k\}$ have controllable loads. The solution of the optimization problem~\eqref{eq:optk} is the load-sharing vector $\a^* = (\widetilde{\a}^* \, \bm{0})$, with
\begin{equation}
\label{eq:ak}
	\widetilde{\a}^* =  \frac{1}{t_B} (L^+_{B})^{-1} \bmo +  \left (I_k - \frac{1}{t_B} (L^+_{B})^{-1}  \bJ \right ) (L^+_{B})^{-1}  P_B^T L^+\frac{\bS \bm{1}}{\sigma^2} \in \R^k,
\end{equation}
where $L_B^+$ is the $k \times k$ principal submatrix of $L^+$, i.e.,~$L_B^+ := P_B^T L^+ P_B$, and $t_B:=\bmo^T (L^+_{B})^{-1} \bmo >0$.\\
If all the nodes with stochastic load fluctuations have controllable loads, i.e., $S \subseteq B$, then
\begin{equation}\label{eq:speccase}
	\widetilde{\a}^* = \frac{\bS \bmo}{\sigma^2},
\end{equation}
and, in particular, $\widetilde{\a}^*_v = 0$ for every $v \in B \setminus S$.
\end{theorem}
The involved expression~\eqref{eq:ak} for the optimal load-sharing factors reflects the interplay that exists between the network structure, the location of the controllable loads $B$, and the correlation structure of the load fluctuations in determining the losses.

In the special case where all the nodes with stochastic load fluctuations have controllable loads there is a nice interpretation for the optimal load-sharing factors: indeed $\a^*_i$ is proportional to how much the stochastic fluctuations of node $i$ contribute in relative terms to the variance of the total mismatch, since
\begin{align}
	\a^*_i &= \frac{1}{\sigma^2} \bm{e}_i^T \bS \bmo = \frac{1}{\sigma^2} \sum_{j=1}^n \bS_{i,j}  = \frac{\mathrm{Var}(\o_i) + \sum_{j \neq i } \mathrm{Cov}(\o_i,\o_j)}{\mathrm{Var}(\sum_{i=1}^n \o_i)} \nonumber\\
	&= \frac{\mathrm{Var}(\o_i) + \sum_{j \neq i } \mathrm{Cov}(\o_i,\o_j)}{\sum_{k=1}^n \left (\mathrm{Var}(\o_k) + \sum_{j \neq k} \mathrm{Cov}(\o_k,\o_j)\right )}. \label{eq:interpretation}
\end{align}

We further remark that in the case of i.i.d.~stochastic fluctuations in all the nodes, the second term in~\eqref{eq:ak} vanishes, since the vector $\Sigma \bm{1}$ lies in the null space of $L^+$ (being a multiple of $\bmo$), and, therefore, the optimal control is equal to
\[
	\widetilde{\a}^* = \frac{ (L^+_{B})^{-1} \bm{1} }{\bm{1}^T (L^+_{B})^{-1} \bm{1}}.
\]
In particular, it does not depend on the variance of the load fluctuations, but only on the network structure and on the location $B$ of the controllable loads, both encoded in the matrix $L^+_B$.

\subsection{Full controllability}
In this subsection we focus on the special case where the load is controllable in every node, i.e.,~$B=V$, which is not covered in Theorem~\ref{thm:optk}. Indeed, the proof method does not work in this scenario due to the non-invertibility of the graph Laplacian $L$, and for this reason is treated separately here.

The problem of minimizing the expected loss $\E \cH(\a)$ when all the nodes have controllable loads can be written as an optimization problem on $\R^n$ with a single constraint, namely
\begin{align}
\label{eq:opt}
	\min_{\a \in \R^n} \quad & \E \cH(\a)\\
	\text{s.t.} \quad & \bm{1}^T \a = 1. \nonumber
\end{align}

As mentioned earlier, thanks to Proposition~\ref{prop:Eha} we can immediately conclude that the optimal load-sharing vector does not depend on the vector $\bmu$, which appears only in the constant term in the equality~\eqref{eq:EH} for the expected loss.

The next theorem derives an analytical expression for the optimal solution of this optimization problem.
\begin{theorem}[Optimal load-sharing between $n$ controllables] \label{thm:optn}
Consider a network $G$ with $n$ nodes and a balanced nominal load profile $\bmo^T \bmu=0$. The solution of the optimization problem~\eqref{eq:opt} is the load-sharing vector $\a^*$ given by
\[
	\a^* = \frac{\bS \bmo}{\sigma^2}.
\]
\end{theorem}
The highlight of this result is that in the scenario where all nodes have controllable loads, the optimal control $\a^*$ does not depend on the graph structure, but only on the covariance structure of the fluctuations. The same interpretation as in the special case $S \subseteq B$ of Theorem~\ref{thm:optk} holds here for the optimal load-sharing factors: $\a^*_i$ is proportional to how much the stochastic fluctuations of node $i$ contribute in relative terms to the variance of the total mismatch, see~\eqref{eq:interpretation}. In particular, when a node $i$ does not have stochastic load fluctuations, then it is optimal not to use the controllable load in that node, since $\a^*_i=0$ in view of the fact that the $i$-th row of $\bS$ is identically zero.

It immediately follows from Theorem~\ref{thm:optn} that when the the stochastic load fluctuations are independent and identically distributed, the optimal load-sharing factors are all equal, namely
\[
	\a^* = \frac{1}{n} \bm{1}.
\]
Furthermore, the expected total loss due to stochastic fluctuations when using the optimal load-sharing factors rewrites as
\[
	\E \cH_s(\a^*) = \frac{1}{2} \left ( \tr(\bS L^+) - \frac{1}{\sigma^2} \bmo^T \bS L^+ \bS \bmo \right ).
\]
In this special case, the fact that $\E \cH_s(\a^*)\geq 0$ can equivalently be proved as follows:
\begin{align*}
	\bmo^T \bS L^+ \bS \bmo = \tr(\bS L^+ \bS J) & \leq \sqrt{ \tr((\bS L^+)^2)\tr((\bS \bJ)^2)} \\
	& \leq \sqrt{ \tr(\bS L^+)^2 \tr(\bS \bJ)^2} = \tr(\bS L^+)\tr(\bS \bJ) = \sigma^2 \tr(\bS L^+),
\end{align*}
where both the inequalities leverage in a crucial way that all the matrices $\bS$, $\bJ$, $L^+$ are positive semi-definite. The first inequality follows from the fact that in the space of positive semi-definite matrices trace is a proper inner-product and thus obeys Cauchy-Schwarz inequality,
\[
	\tr(A B) \leq \sqrt{ \tr(A^2) \, \tr(B^2)} \qquad \forall \, A, B \text{ positive semi-definite matrices.}
\]
The second inequality follows from the fact that $\tr(A^2) \leq \tr(A)^2$ for any positive semi-definite matrix $A$, obtained by applying Cauchy-Schwarz inequality with $A=B$.

\section{Scaling properties of the expected total loss}
\label{sec4}
In this section we explore the relation between the expected total loss and the number of controllable loads. Even if the intuition suggests that the expected total loss should be a decreasing function in the number of controllable loads, this fact may not be true in general, as the total loss depends both on the location of the controllable loads as well as on the the load-sharing factors. For instance, in Section~\ref{sec5} we present a counterexample of a network in which by adding the one controllable load and readjusting the load-share factors to be all equal, the expected total power loss increases.

To get rid of these heterogeneities and obtain a more transparent result for the impact of the number of controllable loads, we calculate the expected total loss for a fixed number of controllable loads \textit{averaging} on all their possible locations and assuming they share equally the load. Consider an integer $1\leq k \leq n$ and denote by $\Ak \subset \R^n$ the collection of load-sharing vectors with exactly $k$ non-zero identical entries (and thus equal to $1/k$, in view of~\eqref{eq:asum}), namely
\[
	\Ak:=\left \{ \frac{1}{k}\sum_{i \in B} \bm{e}_i ~:~ B \subseteq V, \, |B|=k \right \},
\]
where $\bm{e}_i \in \R^n$ is the vector with the $i$-th entry equal to $1$ and zero elsewhere.

Let $\cH_k$ denote the expected total loss due to stochastic load fluctuations averaging over all their possible locations of $k$ controllable nodes share equally the load, i.e.,
\[
	\cH_k := \frac{1}{|\Ak|} \sum_{\a \in \Ak} \E \cH_s(\a).
\]
This average consists of $|\Ak|= \binom{n}{k}$ terms, one for each possible arrangements of $k$ controllable nodes in a network with $n$ nodes. The following theorem states an explicit expression for $\cH_k$ that makes the dependence on the number of controllable node $k$ very explicit, showing in particular that $\cH_k$ is, up to a constant, proportional to $1/k$.

\begin{theorem}[Average total loss with $k$ controllable loads]\label{thm:aveHk}
Consider a network $G$ of $n$ nodes with balanced nominal profile load, i.e.,~$\bm{1}^T \bmu=0$. Then,
\[
	\cH_k = C_1 + \frac{C_2}{k},
\]
where $C_1, C_2 \in \R$ are two constants that do not depend on $k$ given by
\[
	C_1=C_1(G):= \frac{1}{2} \tr(\bS L^+) -  \sigma^2 \frac{\tr(L^+)}{2 n (n-1)}
	\quad \text{ and } \quad
	C_2=C_2(G):= \sigma^2  \frac{\tr(L^+) }{2 (n-1)}.
\]
\end{theorem}

Both the constants $C_1(G)$ and $C_2(G)$ depend on the graph structure via $L^+$, on its size $n$, and on the covariance matrix $\bS$. We remark that the constant $C_2$ is always strictly positive, as $\tr(L^+)>0$ ($L^+$ being a positive semi-definite matrix) and $\sigma^2 >0$, in view of~\eqref{eq:trSJ}.

We are interested in understanding how the expected total loss scales for large graphs. Assume we can take a sequence of graphs $\{G_n\}_{n \in \N}$ of growing size, $|V_n|=n$, and of covariance matrices $\{ \bS_n\}_{n \in \N}$ with total variance $\sigma_n^2=\bmo^T \bS_n \bmo$, so that the limit
\[
	\gamma:= \lim_{n \to \infty} \frac{ (n-1) \, \tr(\bS_n L_n^+)}{ \sigma_n^2 \, \tr(L_n^+)}
\]
exists. Note that the fact that inequality $\cH_{n} \geq 0$ holds for every $n \in \N$ (see Proposition~\ref{prop:Eha}) guarantees that $\gamma \geq 0$. Under these assumptions, Theorem 1 readily implies that as the graph size $n$ grows large the relative gain of having $k$ controllable loads with respect to a single one scales as
\[
	\lim_{n \to \infty} \frac{\cH_k}{\cH_1} 
	= \left (\frac{1}{1+\gamma^{-1}} \right )+  \left (\frac{1}{1+\gamma} \right )\frac{1}{k}.
\]
In the scenario where the load fluctuations are independent and identically distributed, \textit{regardless of the graph structure}, this asymptotic scaling reads
\begin{equation}
\label{scalinglaw}
	\lim_{n \to \infty} \frac{\cH_k}{\cH_1} = \frac{1}{2} + \frac{1}{2 k}.
\end{equation}
as
\[
	\gamma = \lim_{n \to \infty} \frac{ (n-1) \, \tr(\bS_n L_n^+)}{ \sigma_n^2 \, \tr(L_n^+)} =  \lim_{n \to \infty} \frac{(n-1) \, \sigma_n^2 \, \tr(L_n^+)}{n \, \sigma_n^2 \, \tr(L_n^+)} =  \lim_{n \to \infty} \frac{n-1}{n} =1.
\]

In Subsection~\ref{sub:scaling} we show numerically that the way the expected total loss scales on average with the number of controllable loads as stated in Theorem~\ref{thm:aveHk} is pretty accurate in general, even without averaging over all possible locations of the controllable loads.

\section{Extensions}
\label{newsec}
\subsection{Generalized model}
The optimization problems \eqref{eq:optk} and \eqref{eq:opt} considered in Section~\ref{sec3} focus on the network efficiency, but do not prevent the excessive usage of specific controllable loads nor account for economic factors. In this section we present a more general optimization problem that addresses both these issues and that further includes additional constraints typical of the OPF problem, which network operators use to set the operational setpoints of power systems every 5-15 minutes.

Let $P$ be positive definite $n\times n$ matrix, $R \in \R^{n \times n}$, and $\bm{q} \in \R^n$. Introducing a non-negative real number $\xi \geq 0$, we can consider the following generalized optimization problem:
\begin{subequations}\label{eq:opt_p}
\begin{align}
	\min_{\a \in \R^n} & \quad \E \cH(\a) + \xi \, \a^T P \a \label{eq:opt_p1} \\
	\quad \text{s.t.} & \quad \bm{1}^T \a = 1, \label{eq:opt_p2}\\
	&\quad \a_{v} =0, \quad \forall\, v \in V \setminus B, \label{eq:opt_p3}\\
	&\quad R \a \leq \bm{q} \label{eq:opt_p4}\\
	&\quad |\bm{f}_\ell(\a)| \leq \bm{f}_\ell^\mathrm{max}, \quad \forall\, \ell \in E. \label{eq:opt_p5}
\end{align}
\end{subequations}
This optimization problem differs from that in~\eqref{eq:optk} in two ways, namely its objective function~\eqref{eq:opt_p1} has an extra term and it has two new sets of constraints~\eqref{eq:opt_p4}-\eqref{eq:opt_p5}.

The term $\a^T P \a$ can be seen as a \textit{penalty} or \textit{cost} term and the scalar $\xi \geq 0$ weights its relative importance with respect to the average total loss $\E \cH(\a)$. More specifically, by taking $P$ to be a diagonal matrix with non-negative entries, the additional term $\a^T P \a = \sum_{i=1}^n P_{i,i} \a_i^2$ in the objective function penalizes an excessive usage (i.e., large load-sharing factor) of the controllable loads with large coefficients $P_{i,i}$. This means that we can model less flexible or more costly controllable loads, by tuning the corresponding terms $P_{i,i} >0$ accordingly.


The additional constraint~\eqref{eq:opt_p4} can be used to set any desired upper and/or lower bounds on the load-sharing factor of any particular node. In particular, we could restrict the load-sharing factors for some of the nodes $v \in B$ to be non-negative, $\a_v \geq 0$, or their absolute values $|\a_{v}| < \varepsilon_v$ with ad-hoc constants $\varepsilon_v >0$.

Lastly, the constraint~\eqref{eq:opt_p5} captures the physical limits of the lines, imposing the power flow $\bm{f}_\ell(\a)$ on any given network line $\ell \in E$ stays below the corresponding \textit{capacity} $\bm{f}_\ell^\mathrm{max} >0$ of that line. As outlined in Section~\ref{sec2}, using the linear approximation~\eqref{eq:fLp}, the power flows $\bm{f}(\a) = \Lambda L^+ \pp(\a)$ are linear in the load profile $\pp(\a)$ and, ultimately, linear also in the load-sharing factors $\a$.

Clearly~\eqref{eq:opt_p} is still a convex optimization problem: all the constraints are still linear in $\a$ and the objective function rewrites as a quadratic form with the matrix $L^+ + \xi P$ appearing in the leading term that is positive definite in view of the assumption made on the matrix $P$ and Proposition~\ref{prop:Eha}. It is impossible, however, to derive the optimal control $\a^*$ in closed form for the generalized problem~\eqref{eq:opt_p} and thus we cannot explore its structure as we did in the previous sections, but we present some numerical results in next section.

Nonetheless, we argue that the original problem~\eqref{eq:optk} still gives valuable insight as the line capacity limits appearing in~\eqref{eq:opt_p5} are often redundant. Indeed, the OPF problem automatically sets a nominal setpoint $\bmu$ so that all the nominal power flows $\bm{f}$ are within their limits and the optimal control $\a^*$ should decrease the largest line flows (in absolute value) $\bm{f}(\a)$ even further, as it is intuitive from the expression $\cH(\a)=\frac{1}{2} \sum_{\ell \in E} \beta_\ell^{-1} \bm{f}_\ell^2(\a)$, cf.~Section~\ref{sec2}.

Moreover, even if we set a specific load-sharing factor $\a_v$ to be non-negative or smaller than $\varepsilon_v$ with constraint~\eqref{eq:opt_p4}, the controllable load in that node should still be capable of both storing and to outputting a possibly very large amount of power, as we outlined at the end of Subsection~\ref{sub21}. Alternatively, a chance-constraint for the stochastic effort $\a_i \sum_j \o_j$ of node $i$ could be included in the optimization problem like it has been done in~\cite{BCH14}. However, additional assumptions on the multivariate distribution of $\o$ must be made to rewrite the constraint so that the resulting optimization problem is still convex. Other chance-constraints for line limits could be obtained using concentration inequalities as in~\cite{Nesti2017} or using decay rates and large deviation theory as in~\cite{Nesti2018}.

We conclude by deriving the optimal control $\a^*$ when considering the optimization problem with augmented objective function~\eqref{eq:opt_p1} in the special case where all the nodes have controllable loads, i.e., $B=V$, but ignoring the constraints~\eqref{eq:opt_p4}-\eqref{eq:opt_p5}. In this scenario, the optimal solution can be still calculated in closed form as
\[
	\a^* = \Big (\frac{\sigma^2}{2} L^+ + \xi P \Big)^{-1} \left(I + \frac{1}{t} \bJ \Big (\frac{\sigma^2}{2} L^+ + \xi P \Big)^{-1} \right)  L^+ \bS \bm{1} +   \frac{1}{t} \Big (\frac{\sigma^2}{2} L^+ + \xi P \Big)^{-1} \bm{1},
\]
with $t:= \bm{1}^T \Big (\frac{\sigma^2}{2} L^+ + \xi P \Big)^{-1} \bm{1} >0$. This expression can be obtained in a similar way of that in Theorem~\ref{thm:optk}, leveraging the fact that the matrix $L^+ + \xi P$ is positive definite.

\subsection{Variance of the aggregated total loss}
It could be of interest choosing the load-sharing factors $\a_1,\dots,\a_n$ not only to minimize the expected aggregated total loss $\E \cH(\a)$, but also its \textit{variance} $\mathrm{Var} (\cH(\a))$. Aiming to possibly include this in our optimization, we need a closed-form expression for the variance of a quadratic form of a random vector. Unfortunately, this is known only in the case in which the random vector has a multivariate Gaussian distribution. Under this additional assumption, the variance of the aggregated total loss can be calculated explicitly as follows
\begin{align*}
	\mathrm{Var} (\cH(\a)) &= 
	 \mathrm{Var} (\cH_s(\a))  = \mathrm{Var} \left (\frac{1}{2} \o^T \Sa^T L^+ \Sa \o + \bmu^T L^+ \Sa \o \right )\\
	 &= \frac{1}{4} \mathrm{Var} \left (\o^T \Sa^T L^+ \Sa \o \right ) +  \mathrm{Var} \left (\bmu^T L^+ \Sa \o \right ) + \mathrm{Cov}\left  (\o^T \Sa^T L^+ \Sa \o, \bmu^T L^+ \Sa \o \right ) \\
	&= \frac{1}{2} \mathrm{tr} \left ( (C_{\a}^T L^+ C_{\a} \Sigma)^2\right ) + \bmu^T L^+ C_{\a} \Sigma C_{\a}^T L^+ \bmu.
\end{align*}
The last step follows from standard results for quadratic forms of multivariate Gaussian random variables~\cite[Theorems 5.2c and 5.2d]{Rencher2008} and the fact that $\E \o = \bm{0}$.

Following similar steps to those in the proof of Proposition~\ref{prop:Eha} in Appendix~\ref{app:a}, the expression for $\mathrm{Var} (\cH(\a))$ can further rewritten as a polynomial of degree four in the load-sharing factors $\a_1,\dots,\a_n$. This new term could potentially be included in the objective function of the optimization problem~\eqref{eq:opt_p}, but the resulting problem might is not guaranteed to still be convex in general.

\section{Numerical examples}
\label{sec5}
In this section we present some numerical results. First of all, in Subsection~\ref{sub61}, we compare the performance of our static optimal load-sharing factors $\a^*$ with that of idealized real-time load-sharing factors $\a(\o)$ that dynamically respond to fluctuations. Then, in Subsections~\ref{sub:covariance} and~\ref{sub:position} we study the impact of the covariance matrix and of the relative position of the nodes affected by fluctuations and those housing controllable loads. We present two examples to illustrate that (i) optimal load-sharing factors have opposite signs in Subsection~\ref{sub:negative} and that (ii) expected total loss is not always decreasing in the number of controllable loads in Subsection~\ref{sub:nonmonotone}. Lastly,  in Subsection~\ref{sub:scaling} we corroborate the accuracy of the scaling limit derived in Section~\ref{sec4}.
\subsection{Real-time adaptive load-sharing factors}
\label{sub61}
We briefly consider here the scenario in which is possible to (i) have exact knowledge about the realized fluctuations $\o$ and (ii) to have instantaneous coordination between controllable loads. In this setting, a centralized controller could dynamically change the load-sharing factors depending on the realization of the noise $\o$, trying to minimize the total loss. In this case, the total aggregated power loss $\widetilde{\cH}(\a,\o) = \frac{1}{2} (\bmu - \Sa \o)^T L^+ (\bmu - \Sa \o)$ is not a random variable, but simply a quadratic form in $\a$. Thus, for every realization of the noise $\o$ we can find the optimal load-sharing factors $\a^*(\o) \in \R^n$ that solve
\begin{align}
	\min_{\a \in \R^n} \quad & \widetilde{\cH}_s(\a,\o) \nonumber \\
	\text{s.t.} \quad &  \bm{1}^T \a = 1, \label{eq:opt_realized}\\
	\quad & \a_{v} =0, \quad \forall\, v \in V \setminus B. \nonumber
\end{align}
From our previous analysis, it is easy to see that this is still a convex optimization problem in $\a$ with a quadratic objective function and linear equality constraints. Its solution can be obtained from the corresponding Karush-Kuhn-Tucker (KKT) conditions
\[
	\begin{cases}
		-(\bm{1}^T \o) L^+ (\bmu + \o - (\bm{1}^T  \o) \a ) + \pi \bm{1} + \bm{\nu} = \bm{0}\\
		\bm{1}^T \a = 1,\\
		\a_{v} =0, \quad \forall\, v \in V \setminus B,
	\end{cases}
\]
where $\pi \in \R$ is the dual variable corresponding to the constraint $\bm{1}^T \a = 1$ and $\bm{\nu} \in \R^n$ is the vector whose non-zero entries are equal to the dual variables of the constraints $\a_{v} =0$, i.e., $\bm{\nu}_v = \nu \in \R$ if $v \in V\setminus B$ and $0$ otherwise.
In the special case $B=V$, the vector $\bm{\nu}$ is equal to $\bm{0}$ and the system of equations above can be solved explicitly leveraging the spectral structure of $L^+$ to obtain
\[
	\a_i = (\bm{1}^T \o)^{-1} (\bmu_i + \o_i) = \frac{\bmu_i + \o_i}{\sum_{j=1}^n \o_j},
\]
This result corresponds to the unrealistic scenario in which the response of the controllable loads is such that the net load profile $\bm{p}_i(\a)$ becomes zero in every node (and no power flows in the lines).

We focus thus the case where $B \subsetneq V$ and compare numerically the performance of the optimal static load-sharing factors $\a^*$ against the dynamically changing $\a^*(\o)$ for the same realizations of the fluctuations $\o$. We consider the IEEE 14-bus test network, a standard test case used for simulations in the power grid literature~\cite{Zimmerman2011}, with correlated noise, see Figure~\ref{fig:dyn}.
\begin{figure}[!h]
	\centering
	\subfloat[][The topology of the IEEE 14-bus test network]{\includegraphics[scale=0.24]{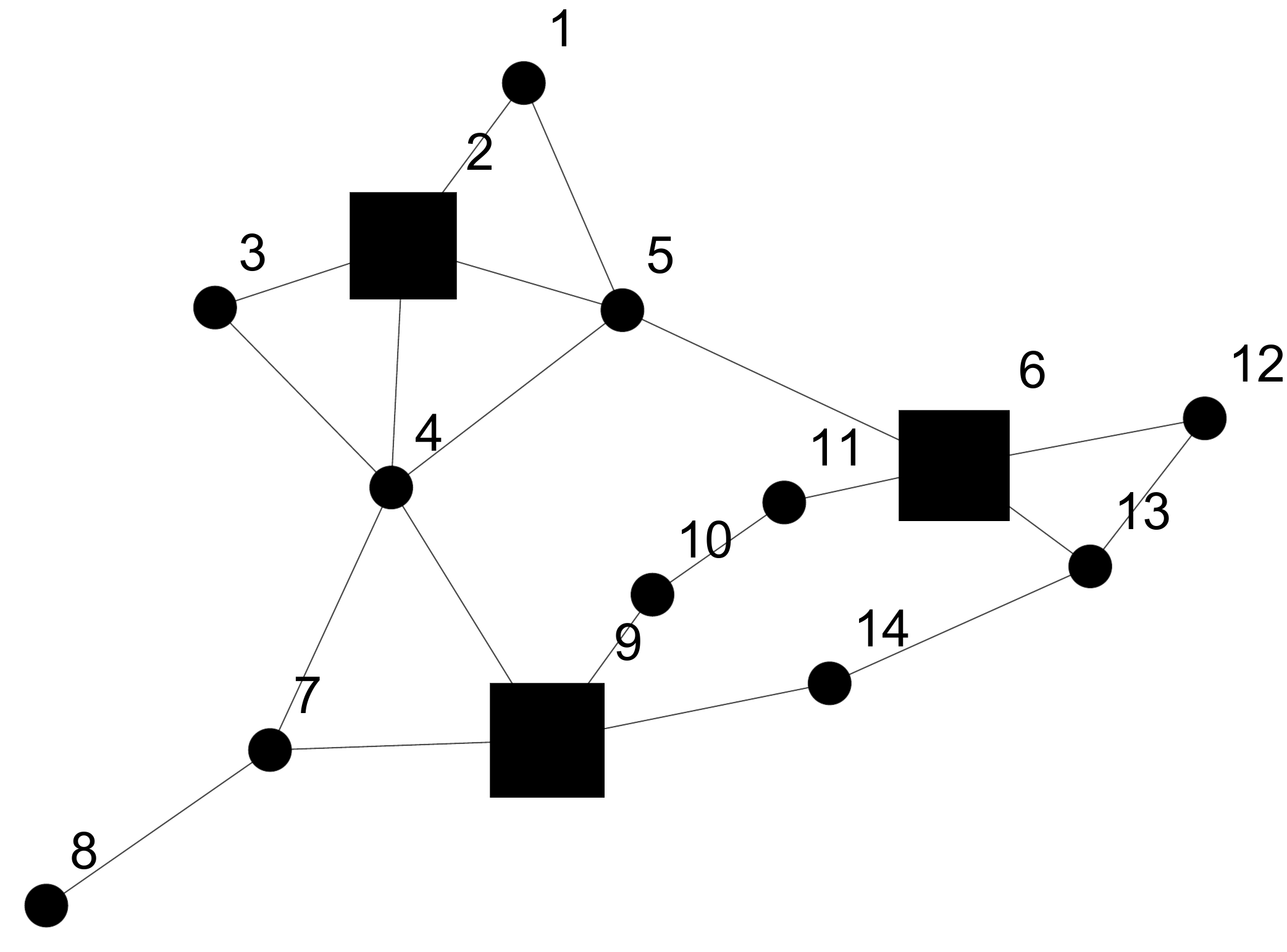}}
	\hspace{0.35cm}
		\subfloat[][Scatter plot of the correlation matrix $\bS$]{\includegraphics[scale=0.5]{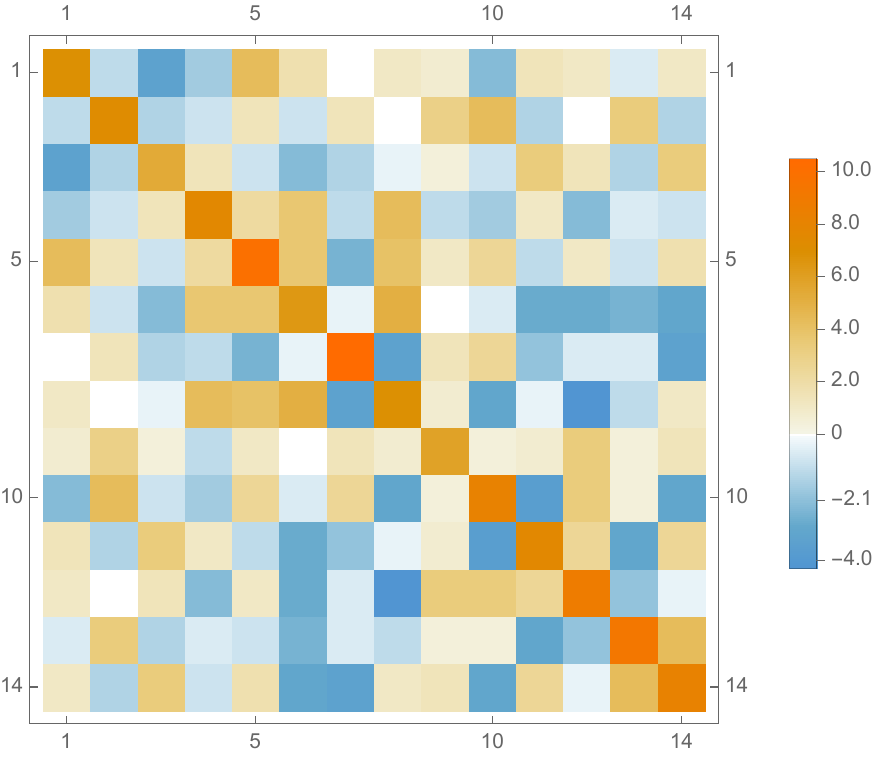}}
	\vspace{0.1cm}
	\caption{The IEEE 14-bus test network with three controllable loads $B=\{2,6,9\}$ (depicted as squares) and all 14 nodes affected by stochastic fluctuations, i.e., $S=V$, with correlations are captured by the matrix $\bS$}
	\label{fig:dyn}
\end{figure}
\FloatBarrier
The resulting static optimal control for this network and correlation structure is $\widetilde{\a}^*=(\widetilde{\a}^*_2,\widetilde{\a}^*_6,\widetilde{\a}^*_9)=(0.3510,0.2805,0.3685)$. We sample $100$ realization of $\o$ from a multivariate Gaussian distribution with mean $\bm{0}$ and correlation matrix $\bS$ and for each of them we calculate the total aggregated loss using $\widetilde{\a}^*$ and finding the optimal dynamic load-sharing factors $\widetilde{\a}^*(\o)$. For these realization, the average total loss using the dynamical control is equal to $18250.8$, while it result in a $20\%$ higher average total loss, namely $21985.1$, when using the static control. However, as shown in Figure~\ref{fig:trace} below, the dynamic control $\widetilde{\a}^*(\o)$ takes incredibly large (both positive and negative) values, which are highly unrealistic to be implemented in practical settings.
\begin{figure}[!h]
	\centering
	\subfloat[][$\widetilde{\a}^*_2(\o)$
	]{\includegraphics[scale=0.415]{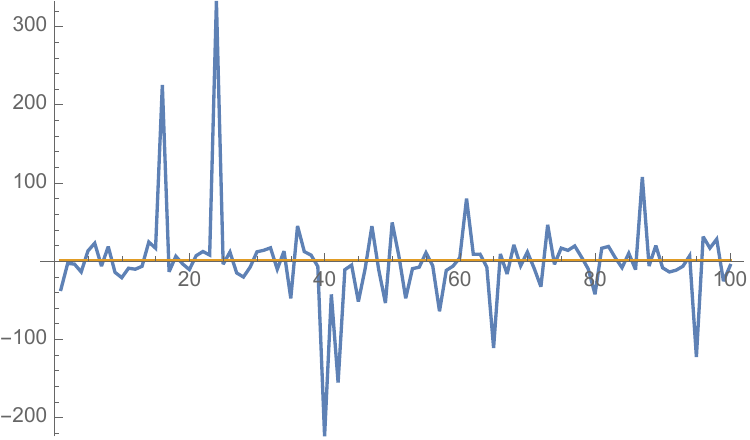}}
	\hspace{0.2cm}
	\subfloat[][$\widetilde{\a}^*_6(\o)$
	]{\includegraphics[scale=0.415]{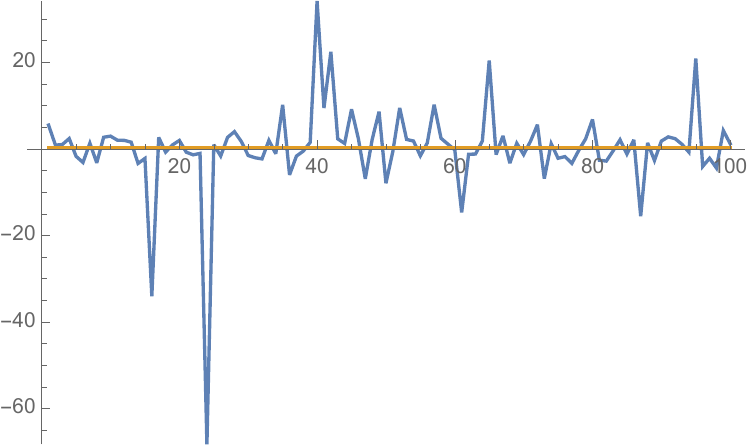}}
	\hspace{0.2cm}
	\subfloat[][$\widetilde{\a}^*_9(\o)$
	]{\includegraphics[scale=0.415]{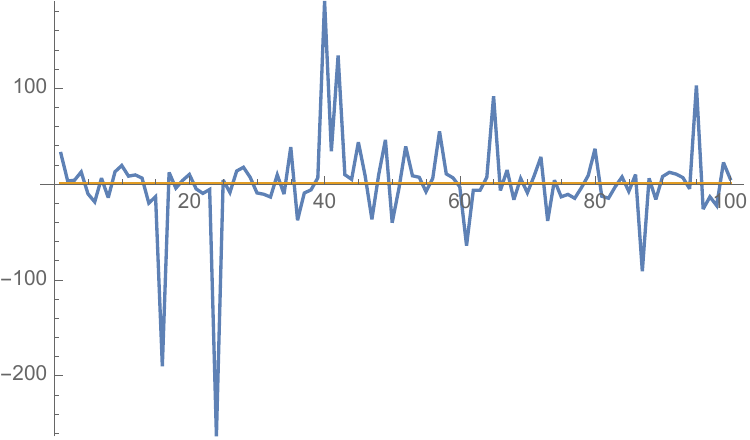}}
	\vspace{0.1cm}
	\caption{Traces of the optimal dynamic load-sharing factors 
	for the three controllable loads $B=\{2,6,9\}$ in the IEEE 14-bus test network of Figure~\ref{fig:dyn} over $100$ realizations of the fluctuations $\o$ }
	\label{fig:trace}
\end{figure}
\FloatBarrier

\subsection{Impact of covariance structure}
\label{sub:covariance}
The next two figures illustrate how the covariance structure influence the optimal control for a small network of $14$ nodes, the IEEE 14-bus test network. Figure~\ref{fig:ieee14}(a) covers the case in which all the nodes have stochastic load fluctuations, i.e.,~$S=V$, while in Figure~\ref{fig:ieee14}(b) we present a scenario where $S\subsetneq V$, where only the nodes in black are affected by fluctuations. In both figures the nodes with controllable loads are $B=\{2,6,9\}$ and are drawn as squares.
\begin{figure}[!h]
	\centering
	\subfloat[][All nodes affected by stochastic fluctuations, $S=V$]{\includegraphics[scale=0.22]{ieee14_s1.pdf}}
	\hspace{0.35cm}
		\subfloat[][Only $9$ nodes (in black) have stochastic fluctuations]{\includegraphics[scale=0.22]{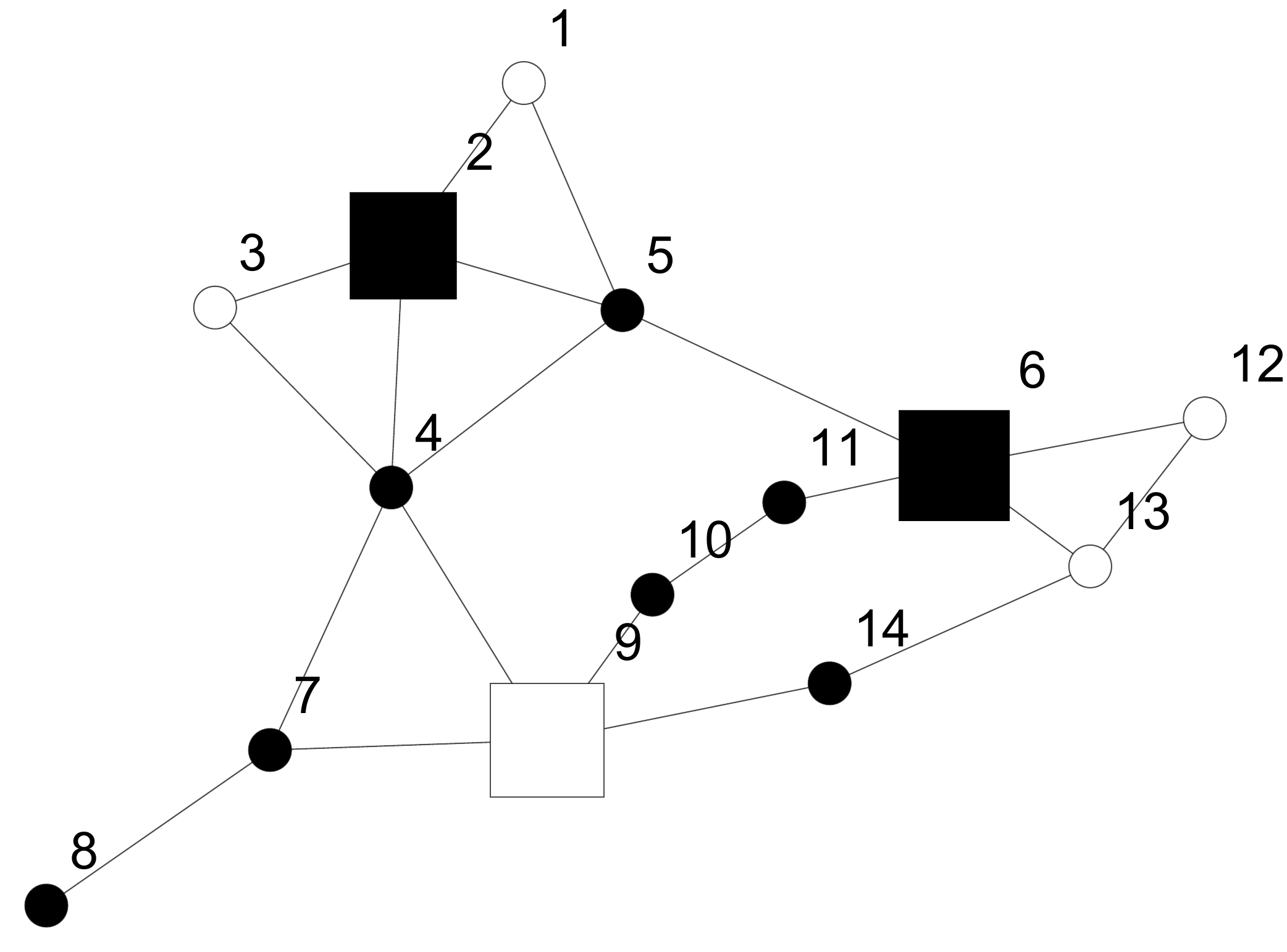}}%
	\vspace{0.1cm}
	\caption{The IEEE 14-bus test network with three controllable loads $B=\{2,6,9\}$ (depicted as squares)}\label{fig:ieee14}
\end{figure}
\FloatBarrier
For these two networks, we compare the optimal static control $\a^*$ for the three controllable loads in the various cases in which we vary the covariance structure of the noise, see Figure~\ref{fig:barchart}.
The correlation matrix $\bS_\mathrm{c}$ in Figure~\ref{fig:covariance}(c) has been generated at random, the correlation matrix $\bS_\mathrm{b}$ used in case (b) is obtained using the diagonal entries of $\bS_\mathrm{c}$, and that in case (a) is equal to $\bS_\mathrm{a}=\sigma^2 I_{14}$, with $\sigma^2=\bmo^T \bS_\mathrm{c} \bmo = \bmo^T \bS_\mathrm{b} \bmo = 255.75$. In this way, all the three correlation matrices have been rescaled so that their total variance is equal in all three cases, namely $\sigma_\mathrm{a}^2=\sigma_\mathrm{b}^2 = \sigma_\mathrm{c}^2$. The correlation matrices for the second network are displayed in Figure~\ref{fig:covariance2}; they have been obtained analogously, but only a subset of $|S|=9$ nodes is affected by stochastic fluctuations, the other one have corresponding rows and columns equal to zero.
\begin{figure}[!h]
	\centering
	\subfloat[][Values of $\widetilde{\a}^*$ for the three cases described in Figure~\ref{fig:covariance}]{\includegraphics[scale=0.51]{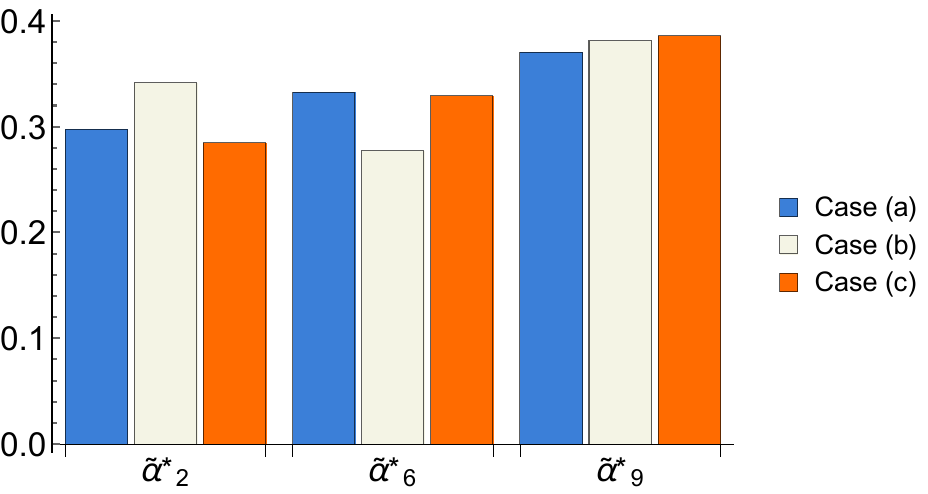}}
	\hspace{0.35cm}
	\subfloat[][Values of $\widetilde{\a}^*$ for the three cases described in Figure~\ref{fig:covariance2}]{\includegraphics[scale=0.51]{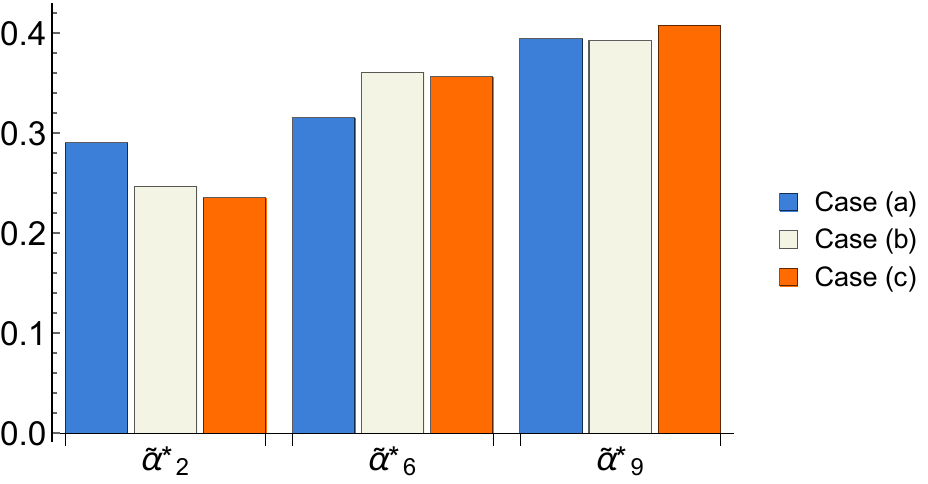}}
	\vspace{0.1cm}
\caption{Values of the optimal load-sharing factors $\widetilde{\a}^*$ with different noise correlation structures. The left chart refers to the network in Figure~\ref{fig:ieee14}(a) with stochastic fluctuations everywhere, the right one to the network with only $9$ stochastic nodes in Figure~\ref{fig:ieee14}(b)}
	\label{fig:barchart}
\end{figure}
\begin{figure}[!h]
	\centering
	\subfloat[][$\bS_\mathrm{a}$ (i.i.d.~load fluctuations)\\ ${\color{white} """} \widetilde{\a}^*=(0.2974, 0.3324, 0.3702)$]{\includegraphics[scale=0.4]{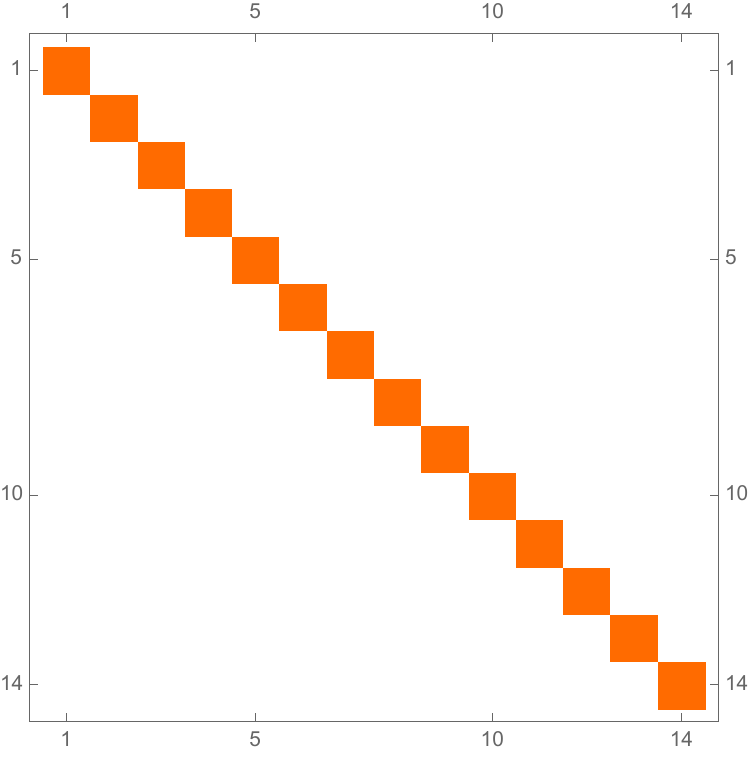}}
	\hspace{0.35cm}
	\subfloat[][$\bS_\mathrm{b}$ (independent but not identically distributed load fluctuations)\\ ${\color{white} """} \widetilde{\a}^*=(0.3414, 0.2774, 0.3812)$]{\includegraphics[scale=0.4]{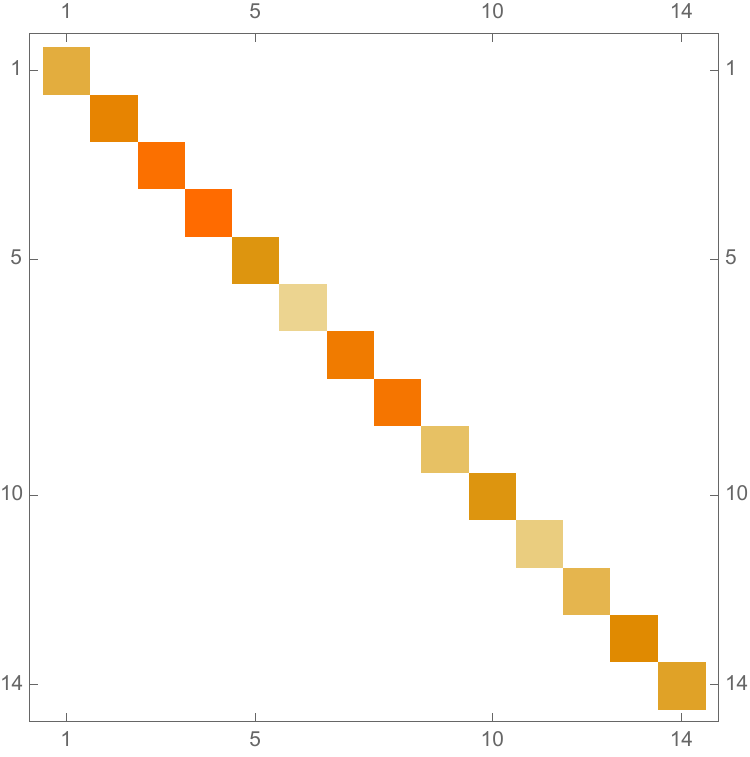}}
	\hspace{0.35cm}
	\subfloat[][$\bS_\mathrm{c}$ (correlated load fluctuations)\\${\color{white} """} \widetilde{\a}^*=(0.2849, 0.3294, 0.3857)$]{\includegraphics[scale=0.4]{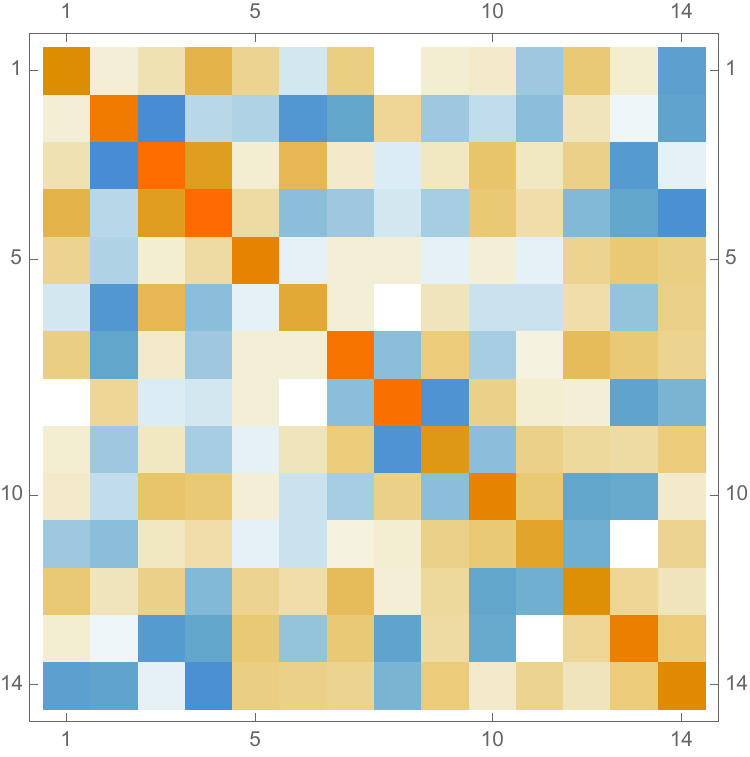}}
	\vspace{0.1cm}
	\caption{Scatter plots of the correlation matrices in three different scenarios and resulting optimal load-sharing factors for the three controllable loads in $B=\{2,6,9\}$ for the network in Figure~\ref{fig:ieee14}(a).}
	\label{fig:covariance}
\end{figure}

\begin{figure}[!h]
	\centering
	\subfloat[][$\bS_\mathrm{a}$ (i.i.d.~load fluctuations)\\ ${\color{white} """} \widetilde{\a}^*=(0.2906, 0.3152, 0.3942)$]{\includegraphics[scale=0.39]{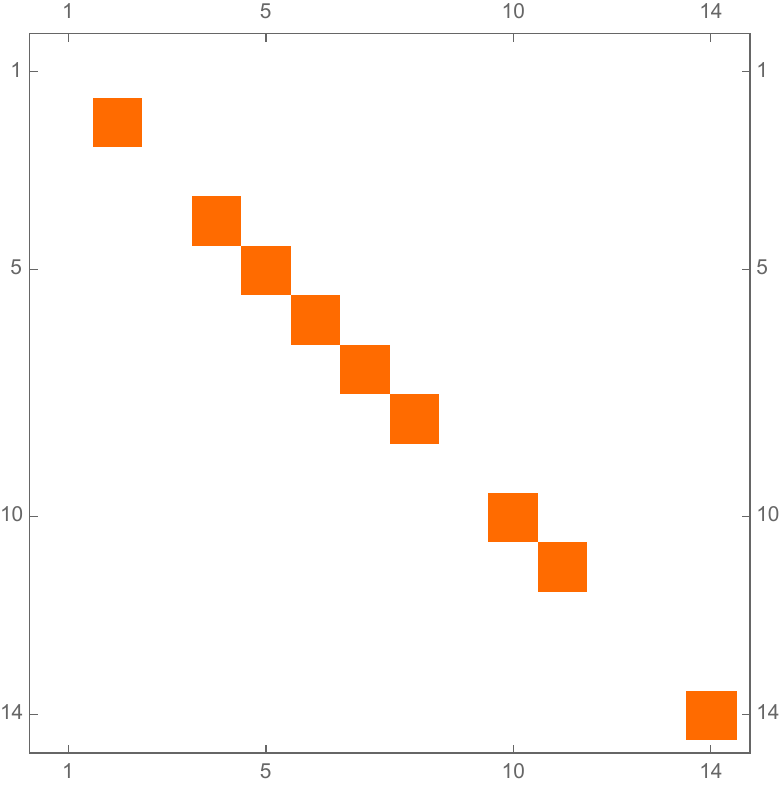}}
	\hspace{0.35cm}
	\subfloat[][$\bS_\mathrm{b}$ (independent but not identically distributed load fluctuations)\\ ${\color{white} """}\widetilde{\a}^*=(0.2470, 0.3607, 0.3923)$]{\includegraphics[scale=0.39]{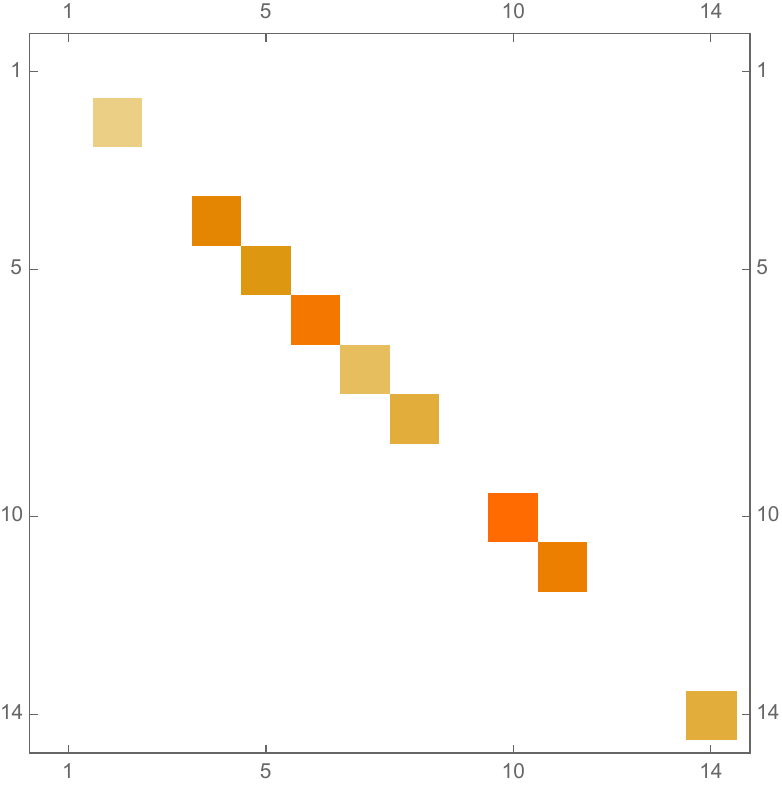}}
	\hspace{0.35cm}
	\subfloat[][$\bS_\mathrm{c}$ (correlated load fluctuations)\\${\color{white} """} \widetilde{\a}^*=(0.2355, 0.3564, 0.4081)$]{\includegraphics[scale=0.39]{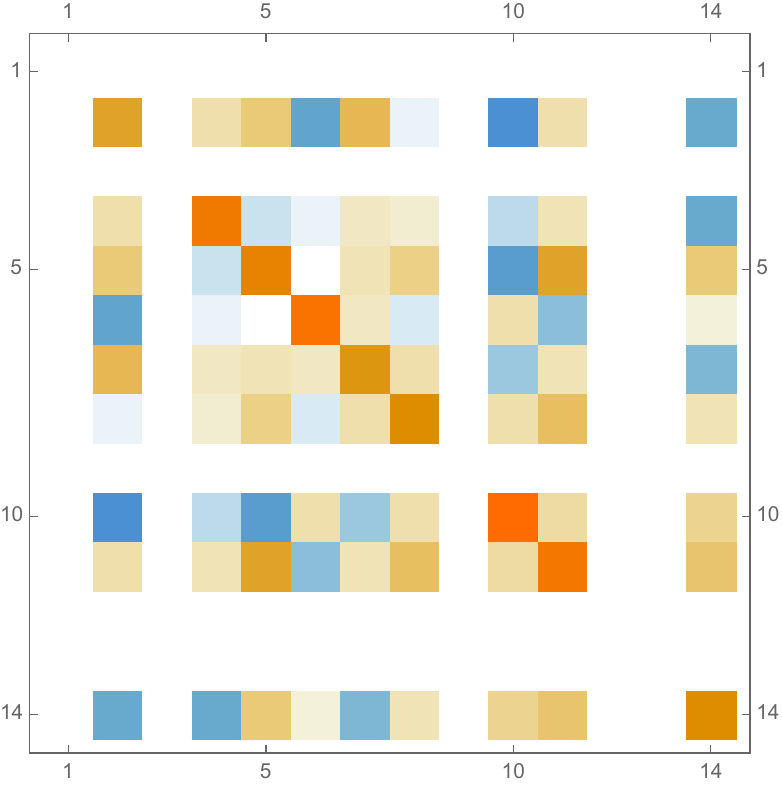}}
\vspace{0.1cm}
	\caption{Scatter plots of the correlation matrices in three different scenarios and resulting optimal load-sharing factors for the three controllable loads in $B=\{2,6,9\}$ for the network in Figure~\ref{fig:ieee14}(b).}
	\label{fig:covariance2}
\end{figure}
\FloatBarrier

\subsection{Relative position of controllable loads and stochastic nodes}
\label{sub:position}
As suggested by Theorem~\ref{thm:optk}, the relative position of the nodes affected by fluctuations and that of the controllable loads play a crucial role in terms of the achievable total power loss, as we will illustrate in the following example on a ring network. Figure~\ref{fig:changingB} visualizes the optimal load-sharing factors in the scenario where there are $|B|=6$ controllable nodes and $|S|=4$ nodes affected by fluctuations. The subset of nodes affected by fluctuation is fixed, $S=\{1,4,7,10\}$, as well as the covariance matrix, but the location of the controllable loads, i.e.,~the subset $B$, changes.
\begin{figure}[!h]
	\centering
	\subfloat[][$S \subset B=\{1,4,5,7,10,11\}$ \\  ${\color{white} """"""} \E \cH_s(\a^*)=2.8825$\\ $ $]{\includegraphics[scale=0.15]{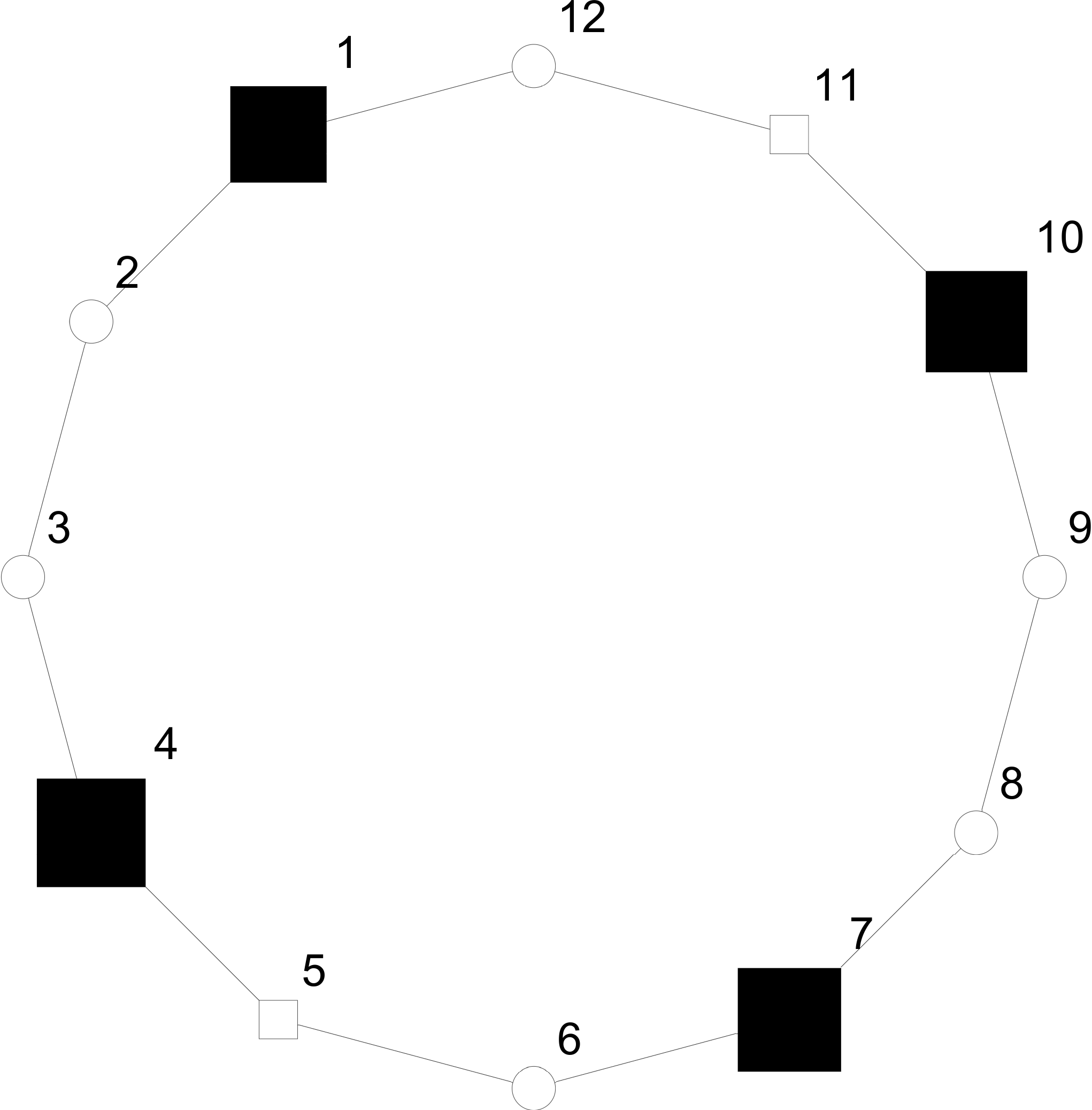}}
	\hspace{0.35cm}
	\subfloat[][$S \not\subset B=\{1,3,5,7,9,11\}$\\ ${\color{white} """"""} \E \cH_s(\a^*)=3.3552$]{\includegraphics[scale=0.15]{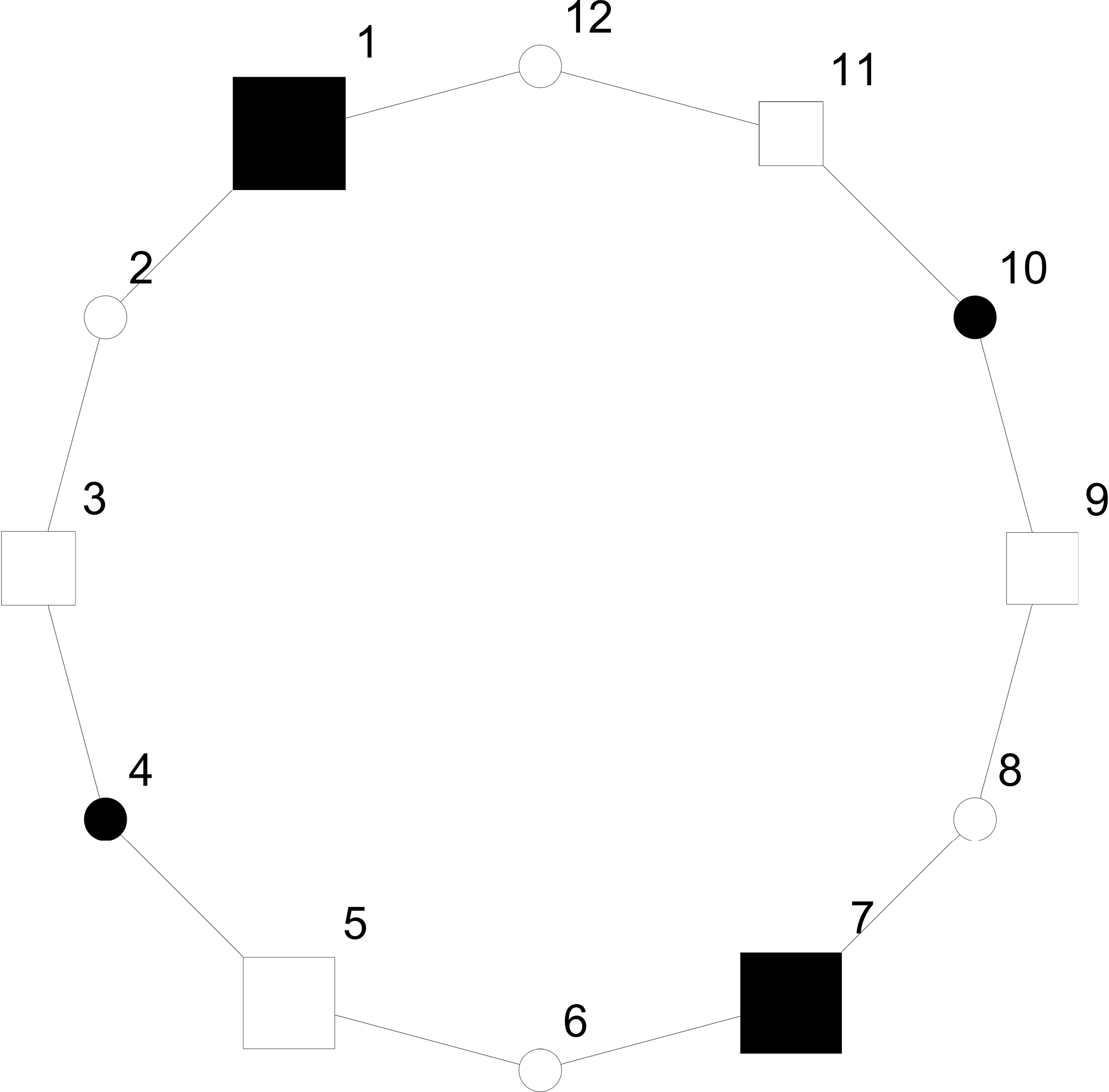}}
	\hspace{0.35cm}
	\subfloat[][$S \subset B^c, \, B=\{2,5,6,8,9,12\}$ \\ ${\color{white} """""""}  \E \cH_s(\a^*)=3.8134$]{\includegraphics[scale=0.15]{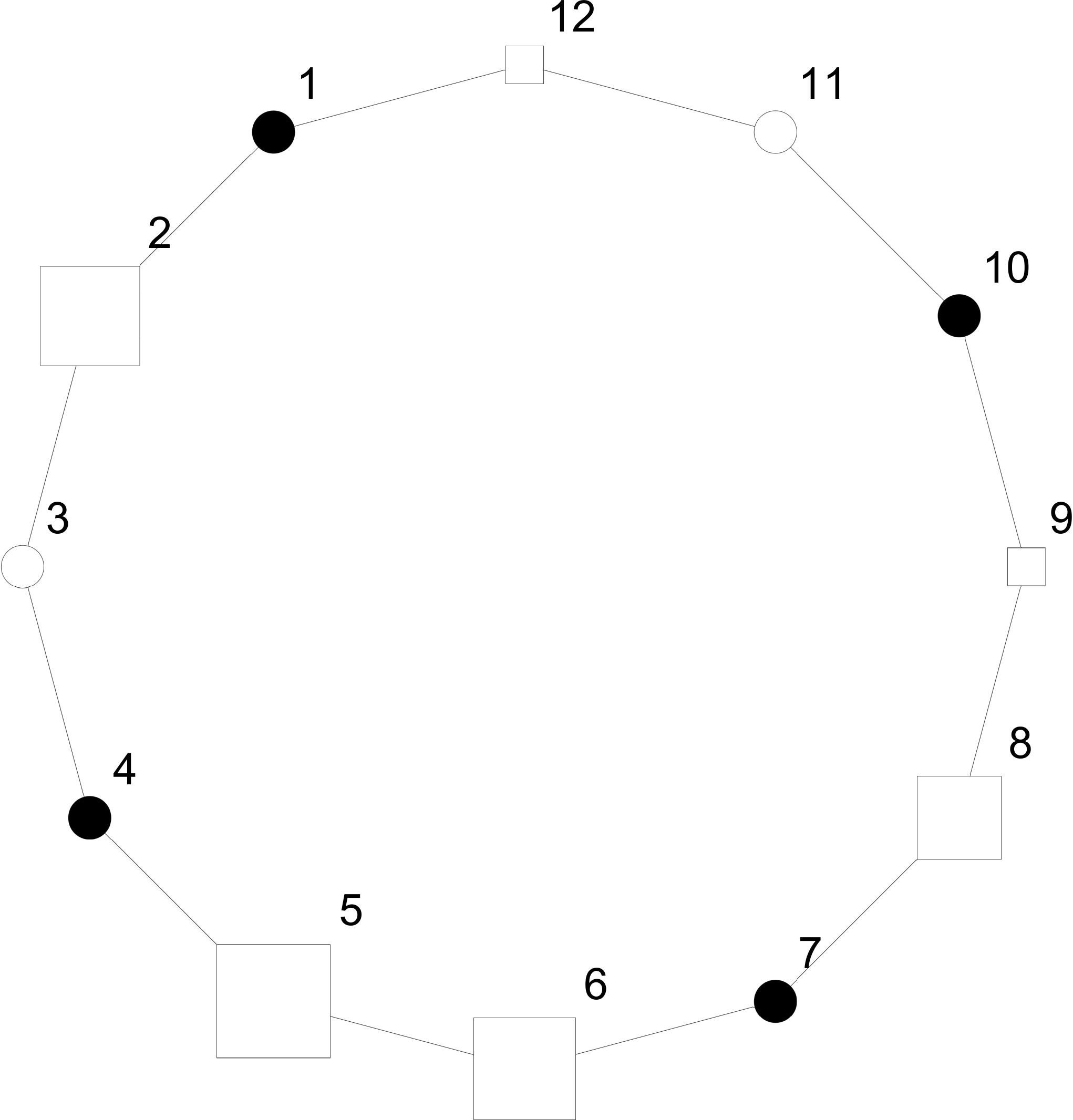}}
	\caption{Different locations of controllable loads $B$ (depicted as squares with area proportional to the optimal load-sharing factors $\a^*$) and corresponding expected total loss.}
	\label{fig:changingB}
\end{figure}
Figure~\ref{fig:changingB}(a) presents the scenario in which $S \subset B$ (i.e., all the stochastic nodes have controllable loads) and the optimal control $\a^*$ for all other nodes in $B \setminus S = \{5,11\}$ is equal to zero, as prescribed by Theorem~\ref{thm:optk}. In the other two cases, in Figures~\ref{fig:changingB}(b) and (c), we picked two different subsets $B$ of controllable loads such that $S \not\subset B$ and $S \subset B^c$, respectively. The corresponding values of the expected loss are higher in these cases than in case (a) and suggest that it may be optimal to place the controllable loads in the nodes affected by stochastic fluctuations. Lastly, note in Figure~\ref{fig:changingB}(b) that the load-sharing factors for node $1$ and $7$ are much larger than the other nodes in $B$, since these two nodes are affected by stochastic fluctuations but the remaining four ones are not. 
\FloatBarrier

\subsection{Negative load-sharing factors}
\label{sub:negative}
The fact that a load-sharing factor is non-negative means that the corresponding node absorbs part of power excess (if $\sum_{i=1}^n \o_i>0$) and balance out shortages (if $\sum_{i=1}^n \o_i<0$). In most of the related work in primary response mechanisms and automatic generation controls for power grids, the participation factors (that our load-sharing factors generalize) are in fact taken to be non-negative, i.e.,~$\a_v \geq 0$ for all $v \in B$. This assumption tacitly implies that all the controllable generators and storage have ``coordinated'' actions, i.e., they either all increase or all decrease their power output.

In our formulation of the optimization problems~\eqref{eq:optk} and~\eqref{eq:opt} we do not make such an assumption and load-sharing factors can also be negative, as long as the condition~\eqref{eq:asum} is met. This is crucial as for certain covariance structures of the load fluctuations (especially when there are strong negative correlations) it is optimal to have negative load-sharing factors in some nodes: we illustrate this fact for a small network illustrated in Figure~\ref{fig:smallnetwork}.
\begin{figure}[!h]
\centering
\includegraphics[scale=0.19]{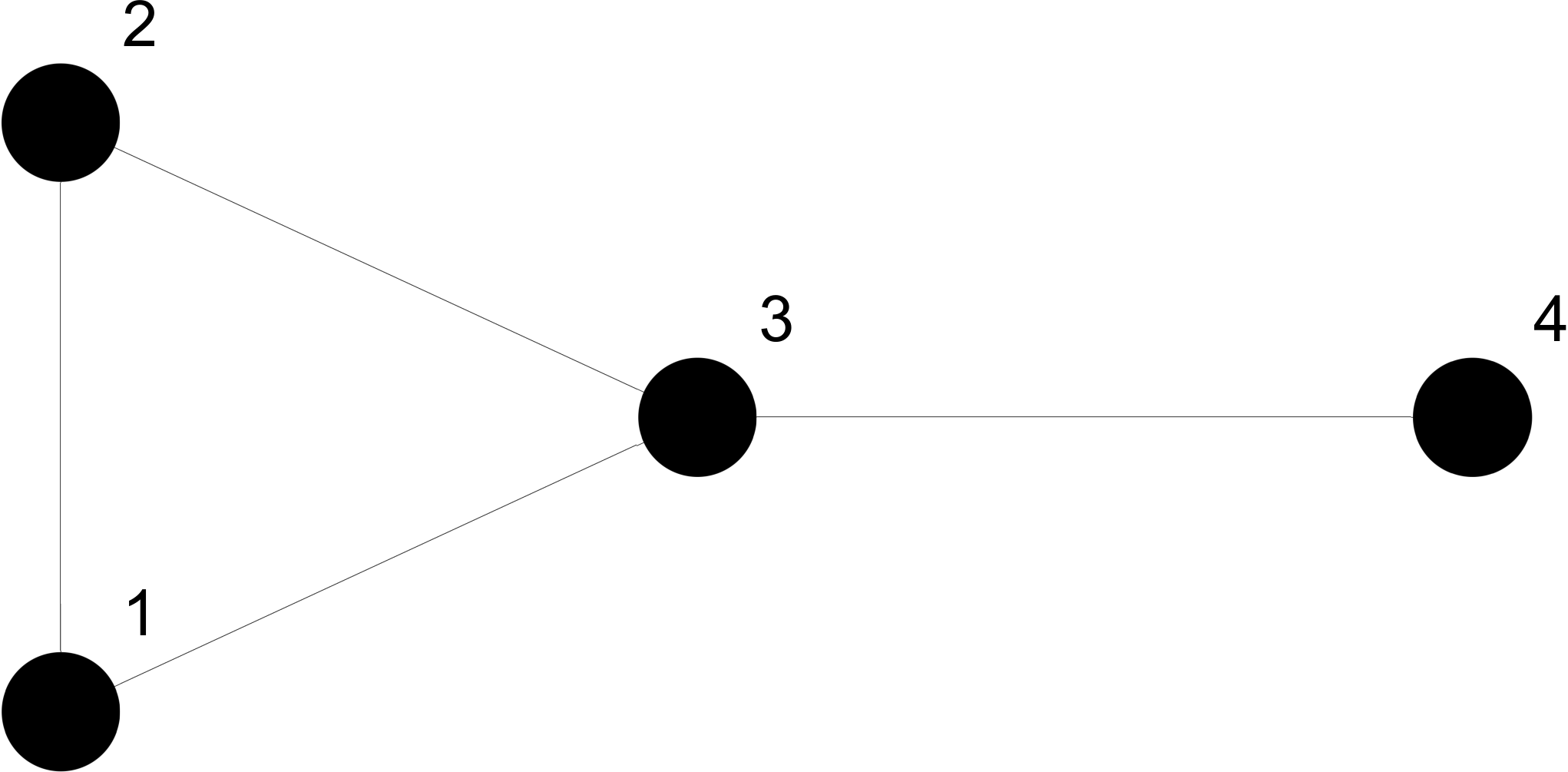}
\vspace{0.2cm}
\caption{A small network modeled by a graph with $n=4$ nodes and $m=4$ edges with unit weights}
\label{fig:smallnetwork}
\end{figure}
\FloatBarrier
Table~\ref{tab:negative} below lists the optimal load-sharing factors $\a^*$ corresponding to different set of controllable loads in the network in Figure~\ref{fig:smallnetwork} where the load fluctuations covariance structure is assumed to be
\[
\Sigma=
\left(
\begin{array}{cccc}
 1 & 0 & 0 & -0.5 \\
 0 & 1 & 0 & -0.5 \\
 0 & 0 & 1 & -0.5 \\
 -0.5 & -0.5 & -0.5 & 1 \\
\end{array}
\right).
\]
The best way for the controllable loads to respond to the negative correlations that the load fluctuations have in this network is having the controllable load in node $4$ taking actions ``mirroring'' those of the other three nodes, in the sense that $\a^*_4 <0$ while the load-sharing factors of the other nodes in $B$ are always positive.
\begin{table}[!h]
\centering
\vspace{-0.25cm}
\begin{tabular}{l|l}
 $\quad \, B$ & $\quad \qquad \a^*(B)$\\
 \hline
$\{1,4\}$& $ (6/5, 0, 0, -1/5)$\\
$\{1,2,4\}$& $ (2/3, 2/3, 0, -1/3)$\\
$\{1,2,3,4\}$& $ (1/2,1/2, 1/2,-1/2)$
\end{tabular}
\vspace{0.2cm}
\caption{Sets of controllable loads for the network in Figure~\ref{fig:smallnetwork} and corresponding optimal load-sharing factors}%
\label{tab:negative}%
\end{table}
\FloatBarrier
\subsection{Non-monotonicity of expected power loss when adding controllable loads}
\label{sub:nonmonotone}
An extra controllable load always reduces the expected total power loss if the corresponding optimal load-sharing vector $\a^*$ is selected, since it corresponds to removing one constraint in the optimization problem~\eqref{eq:optk}. However, if the chosen load-sharing factors of the augmented subset of controllable loads are \textit{not} the optimal ones, adding an extra controllable load does not necessarily reduce the expected total power loss. We illustrate this fact with an example in which the control is always assumed to be equal-share between the controllable loads in $B$, i.e., $\widetilde{\a} = \frac{1}{|B|} \bmo$. Consider the network given in Figure~\ref{fig:notmonotone} and assume that the stochastic loads are i.i.d.~with unit variance.
\begin{figure}[!h]
\centering
\includegraphics[scale=0.22]{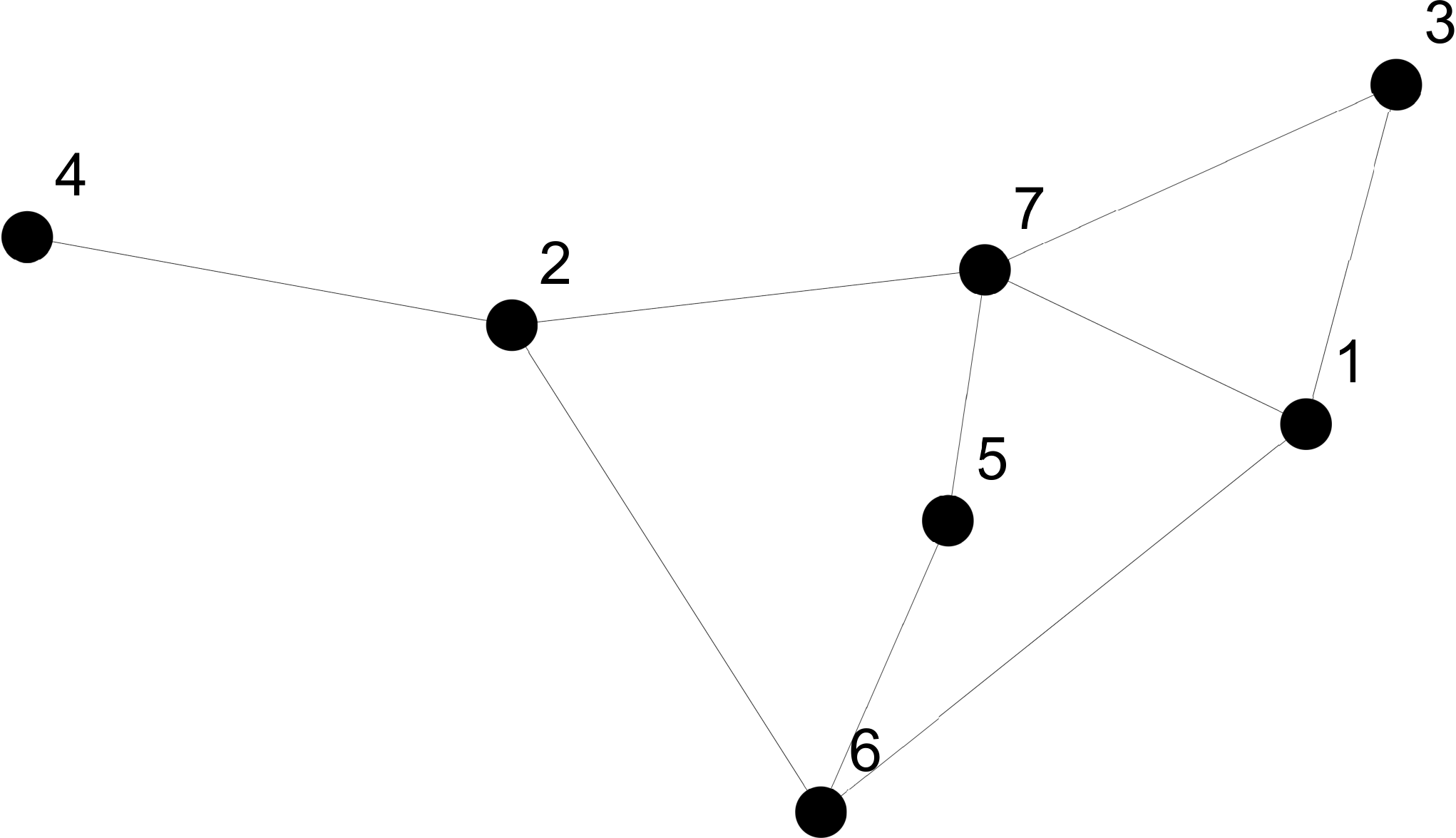}
\vspace{0.2cm}
\caption{A small network modeled by a graph with $n=7$ nodes and $m=9$ edges with unit weights}
\label{fig:notmonotone}
\end{figure}
\FloatBarrier
Table~\ref{tab:notmonotone} below lists the expected total power losses for some subsets $B$ of controllable loads and compares them with those for some augmented subset $B \cup \{4\}$. It is evident that in every one of these case, adding an additional controllable load in node $4$ without optimally readjusting the load-sharing factors result in a higher expected total loss.
\begin{table}[!h]
\centering
\vspace{-0.25cm}
\begin{tabular}{l|c|c}
 $\quad \, B$ & $\E \cH(B)$& $\E \cH(B\cup \{4\})$\\
 \hline
$\{2\}$&  $2.6875$ & $3.0625$ \\
$\{2,5\}$&  $2.0625$ & $2.1319$\\
$\{2,6,7\}$& $1.7986$ & $1.8438$
\end{tabular}
\vspace{0.2cm}
\caption{Expected total power losses for the network in Figure~\ref{fig:notmonotone} assuming equal load-sharing factors for some subsets $B$ of controllable loads and then for the subsets augmented with an extra node, namely $B \cup \{4\}$.}%
\label{tab:notmonotone}%
\end{table}%
\FloatBarrier%

\subsection{Empirical evidence of the scaling law}
\label{sub:scaling}
%
The scaling law derived in Section~\ref{sec4}, despite having been obtained averaging over the possible location of the controllable loads, still give precious insight about how the average total loss decrease with the number of controllable loads. Aiming to corroborate this fact, we consider the IEEE RTS 96-bus test network~\cite{Zimmerman2011} and track the average total loss while adding one by one controllable loads in random locations and assuming equal share among them.

\begin{figure}[!h]
	\centering
	\subfloat[][i.i.d.~load fluctuations]{\includegraphics[scale=0.4]{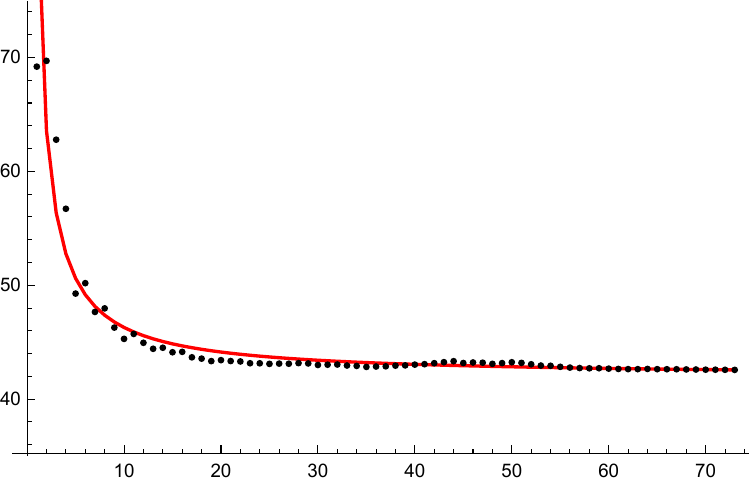}}
	\hspace{0.35cm}
	\subfloat[][Independent but not identically distributed load fluctuations]{\includegraphics[scale=0.4]{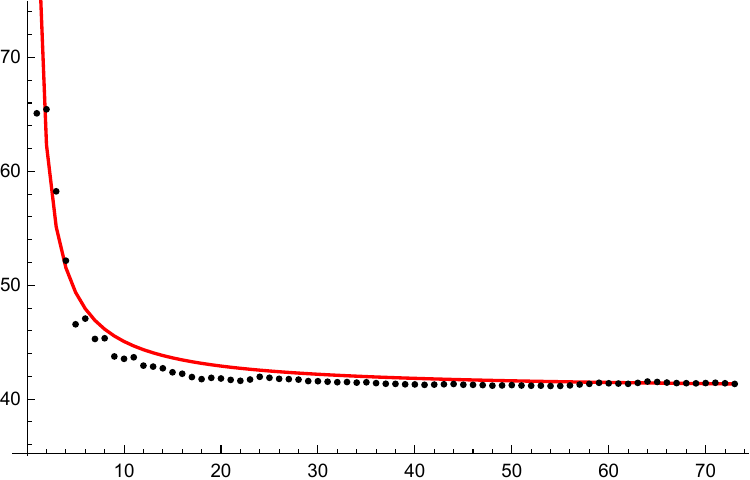}}
	\hspace{0.35cm}
	\subfloat[][Correlated load fluctuations]{\includegraphics[scale=0.4]{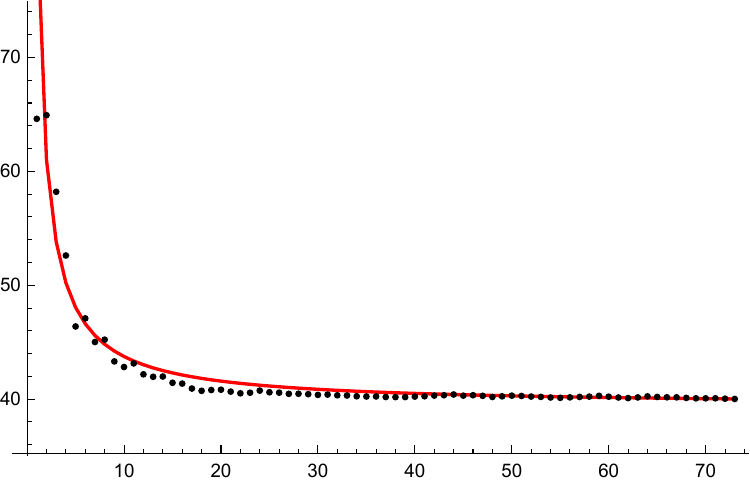}}
	\caption{Theoretical scaling (in red) of the expected total loss as predicted by Theorem~\ref{thm:aveHk} vs.~empirical total loss (in black) while adding controllable loads one by one in random locations.}
	\label{fig:96}
\end{figure}
\FloatBarrier
As illustrated by Figure~\ref{fig:96}, the theoretical scaling for the expected total loss scales with the number of controllable loads stated in Theorem~\ref{thm:aveHk} while averaging on all possible locations is in fact very accurate also for a single instance where new controllable locations are randomly added.


\section{Concluding remarks}
\label{sec6}
In this paper we consider a stochastic lossy transport network in which some nodes have controllable loads and derive a closed-form expression for the optimal control when aiming to minimize the average total loss. The model is inspired by power systems where distributed energy resources can be used as virtual storage to mitigate the fluctuations in the power generated by renewable energy sources and in power demand. Our analysis unveils the complex interplay between the network structure, the location of the controllable loads and the covariance structure of the power fluctuations and gives insight in how much the average total loss can be reduced by adding a given number of controllable loads to the network.

We derived explicit optimal load-sharing factors for controllable loads in various scenarios. Our analysis, even if it uses a stylized mathematical model, suggests that the optimal displacement and operations of distributed energy resources must account for the possible correlations of the power fluctuations. For this reason it complements the recent efforts in the electrical engineering community in upgrading the existing models for power grids to account both for the intrinsic volatility of renewable energy generation and storage capabilities, see e.g.~\cite{Kanoria2011,LinBitar2016a,LinBitar2016b,LinBitar2016c,LorcaSun2015,Lorca2017}.

Lastly, we presented a more general optimization problem that can be instrumental to explore numerically the trade-off between the best operations for the network and the corresponding cost or penalties for the excessive usage of the controllable nodes. This is particularly relevant for the design of primary response mechanisms and automatic generation controls for power grids~\cite{Apostolopoulou2014,Chertkov2017d,Guggilam2017b,Guggilam2017c,LubinDvorkinBackhaus2016,RoaldChertkov2016,Bienstock2016}. 

\subsection*{Acknowledgments}
This research is supported by NWO VICI grant 639.033.413 and NWO Rubicon grant 680.50.1529.

\appendix
\appendixpage

\section{Properties of the total power loss}
\label{app:a}
\begin{proof}{Proof of Proposition~\ref{prop:Eha}.}
Using the expression for $\Sa$ given in~\eqref{eq:Sap}, we can rewrite $\cH(\a)$ as
\[
	\cH(\a) = \frac{1}{2} (\bmu + \Sa \o)^T L^+ (\bmu + \Sa \o)  =  \frac{1}{2} \bmu^T L^+\bmu + \frac{1}{2} \o^T \Sa^T L^+ \Sa \o + \bmu^T L^+ \Sa \o.
\]
Note that one of the terms on the RHS, namely $\frac{1}{2} \bmu^T L^+\bmu$, is not random and does not depend on the control $\a$. Define the random variable
\begin{equation}
\label{eq:Hs}
	\cH_s(\a) :=  \cH(\a) - \frac{1}{2} \bmu^T L^+\bmu = \frac{1}{2} \o^T \Sa^T L^+ \Sa \o + \bmu^T L^+ \Sa \o,
\end{equation}
which describes precisely the contribution of the stochastic fluctuations to the transportation losses. From the fact that $L^+$ is a positive semi-definite matrix it follows that
\begin{equation}
\label{eq:Hspositive}
	\frac{1}{2} \o^T \Sa^T L^+ \Sa \o = \frac{1}{2} (  \Sa \o)^T L^+ (\Sa \o) \geq 0 \quad \forall \, \o \in \R^n.
\end{equation}
Combining~\eqref{eq:Hs}, \eqref{eq:Hspositive}, and the fact that $ \E (\bmu^T L^+ \Sa \o )=   \bmu L^+ \Sa \E \o = \bm{0}$ yields that
\[
	\E \cH_s(\a) = \E \left (\frac{1}{2} \o^T \Sa^T L^+ \Sa \o \right ) \geq 0.
\]
Applying a classic result for quadratic forms of random vector, see e.g.~\cite[Corollary 3.2b.1]{MP92}, we derive
\begin{equation}\label{eq:Ehsp}
	\E \left (\frac{1}{2} \o^T \Sa^T L^+ \Sa \o \right ) = \frac{1}{2} \, \tr (\Sa^T L^+ \Sa \bS).
\end{equation}
Since $\E \o = \bm{0}$, it follows that $ \E (\bmu^T L^+ \Sa \o )=   \bmu^T L^+ \Sa \E \o = \bm{0}$ and thus identity~\eqref{eq:Hs} can be rewritten as
\[
	 \E \cH_s(\a) = \frac{1}{2} \, \tr (\Sa^T L^+ \Sa \bS).
\]

We now derive identity~\eqref{eq:EHs}. Recall the following well-known properties of the trace of matrix:
\begin{itemize}
	\item[\textup{(i)}] The trace is invariant under cyclic permutations, i.e.,~for any $r \in \N$
	\[
		\tr(A_1 \dots A_r) = \tr(A_2\dots A_r A_1) = \dots = \tr(A_r A_1 \dots A_{r-1}).
	\]
	\item[\textup{(ii)}] The trace of a matrix and of its transpose coincide, i.e.,~$\tr(A)=\tr(A^T)$;
	\item[\textup{(iii)}] The trace of the outer product of two vectors is their inner product, namely
	\[
		\tr(\bm{v}\bm{w}^T) = \tr(\bm{v} \otimes \bm{w}) = \bm{v}^T \bm{w}.
	\]
\end{itemize}
First note that we can rewrite
\begin{align}
	\Sa^T L^+ \Sa
	&= (I - \a \, \bm{1}^T)^T L^+ (I- \a \, \bm{1}^T) = (I - \bm{1} \, \a^T ) ( L^+ - L^+ \a \, \bm{1}^T) \nonumber \\
	&= L^+ - L^+ \a \, \bm{1}^T - \bm{1} \, \a^T L^+  + \bm{1} \, \a^T L^+ \a \, \bm{1}^T. \label{eq:sls}
\end{align}
The aforementioned properties of the trace yield
\begin{equation}
\label{eq:traceeq1}
	 \tr(\bS L^+ \a \, \bm{1}^T) \stackrel{\textup{(ii)}}{=} \tr((\bS L^+ \a \, \bm{1}^T)^T) = \tr(\bm{1} \a^T L^+ \bS ) \stackrel{\textup{(i)}}{=}  \tr(\bS \bm{1} \a^T L^+),
\end{equation}
and
\begin{equation}
\label{eq:traceeq2}
	\tr(\bS \bm{1} \a^T L^+) = \tr ( \bS \bm{1} (L^+ \a)^T) = \tr ((\bS \bm{1}) \otimes (L^+ \a)) \stackrel{\textup{(iii)}}{=} (\bS \bm{1})^T (L^+ \a) \stackrel{\textup{(iii)}}{=}  \bm{1}^T \bS L^+ \a.
\end{equation}
By combining all these equalities and exploiting the linearity of the trace operator, we obtain
\begin{align*}
	\tr(\bS  \Sa^T L^+ \Sa)
	& \stackrel{\eqref{eq:sls}}{=} \tr(\bS L^+) - \tr(\bS L^+ \a \, \bm{1}^T) - \tr(\bS \bm{1} \, \a^T L^+)  + \tr(\bS \bm{1} \, \a^T L^+ \a \, \bm{1}^T)\\
	& \stackrel{\eqref{eq:traceeq1}}{=} \tr(\bS L^+) - 2 \cdot \tr(\bS \bm{1} \, \a^T L^+)  + (\a^T L^+ \a) \cdot \tr(\bS \bm{1} \,  \bm{1}^T)\\
	& \stackrel{\eqref{eq:traceeq2}}{=} \tr(\bS L^+) - 2 \cdot (\bm{1}^T \bS L^+ \a )  + \sigma^2 (\a^T L^+ \a),
\end{align*}
where we also used the fact that $\a^T L^+ \a$ is a scalar in the second step and identity~\eqref{eq:trSJ} in the third step.
\end{proof}

\begin{proof}{Proof of Proposition~\ref{prop:monotone}}
Assume that $e=(i,j) \in (V\times V)$ is the edge with weight $\beta>0$ that has been added to $G$ or whose edge weight has been increased by $\beta>0$ and let $\bm{m}_e =( \bm{e}_i - \bm{e}_j) \in \R^n$ be the corresponding non-weighted incidence vector. In both cases the Laplacian matrix of the newly obtained graph $G'$ can be written as
\[
	L_{G'} = L_G + \beta \, \bm{m}_e \bm{m}_e^T
\]
and, using the generalized version of the Sherman-Morrison formula in~\cite{Meyer73}, we get
\[
	L_{G'}^+ = L_G^+ - \frac{1}{\beta^{-1} + \bm{m}_e^T L_G^+ \bm{m}_e}  L_G^+ \bm{m}_e \bm{m}_e^T L_G^+.
\]
We can thus rewrite the total loss corresponding to any net load profile $\pp(\a)$ as
\begin{align*}
	\cH^{G'}(\a) &=  \frac{1}{2} \pp(\a)^T L_{G'} \pp(\a) \\
	&=  \frac{1}{2} \pp(\a)^T L_G^+ \pp(\a) - \frac{1}{2(\beta^{-1} + \bm{m}_e^T L_G^+ \bm{m}_e)} \pp(\a)^T L_G^+ \bm{m}_e \bm{m}_e^T L_G^+ \pp(\a) \\
	&= \cH^G(\a) - \frac{(m_e^T L_G^+ \pp(\a))^2}{2(\beta^{-1}+\bm{m}_e^T L_G^+ \bm{m}_e)},
\end{align*}
and conclude by noticing that $\bm{m}_e^T L_G^+ \bm{m}_e \geq 0$ and $(\bm{m}_e^T L_G^+ \pp(\a))^2 \geq 0$.
\end{proof}

\section{Proof of Theorem~\ref{thm:optk}}
In this proof we use the so-called \textit{block matrix inversion formula}, which is stated in the next lemma.
\begin{lemma}[Block matrix inversion formula]\label{lem:block}
Consider a matrix with the block structure $
	\begin{pmatrix}
	\displaystyle A & B\\
	\displaystyle C & D
	\end{pmatrix}
$. If both $A$ and $D-C A^{-1} B$ are non-singular matrices, then
\[
	\begin{pmatrix}
	\displaystyle A & B\\
	\displaystyle C & D
	\end{pmatrix}^{-1}
	=
	\begin{pmatrix}
	\displaystyle A^{-1} - A^{-1} B (D-C A^{-1} B)^{-1} C A^{-1} & -A^{-1} B (D-C A^{-1} B)^{-1}\\
	\displaystyle -(D-C A^{-1} B)^{-1} C A^{-1} & (D-C A^{-1} B)^{-1}
	\end{pmatrix}.
\]
\end{lemma}

For any vector $\a \in \R^n$ which is such that $\a_{i} =0$ for every node $i \in V \setminus B$, there exists a unique $k$-dimensional vector $\widetilde{\a} \in \R^k$ such that $\a = P_B \widetilde{\a}$. Using this correspondence and the fact that the nominal load profile is balanced, i.e.,~$\bm{1}^T \bmu = 0$, we can rewrite
\[
	\E \cH_s(\a) = \E \cH (P_B \widetilde{\a}) \stackrel{\eqref{eq:EH}}{=}  \frac{\sigma^2}{2} (\widetilde{\a}^T P_B^T L^+ P_B \widetilde{\a}) -  \bm{1}^T \bS L^+ P_B \widetilde{\a}  + \frac{1}{2} \tr(\bS L^+).
\]
Therefore the $n$-dimensional optimization problem~\eqref{eq:optk} rewrites as a $k$-dimensional optimization problem with a single constraint, namely
\begin{align}
\label{eq:optk_rewritten}
	\min_{\widetilde{\a} \in \R^k} \quad & \frac{\sigma^2}{2} (\widetilde{\a}^T P_B^T L^+ P_B \widetilde{\a}) -  \bm{1}^T \bS L^+ P_B \widetilde{\a} \\
	\text{s.t.} \quad &  \bm{1}^T \widetilde{\a} = 1. \nonumber
\end{align}

The matrix $L^+_{B}:= P_B^T L^+ P_B \in \R^{ k \times k}$ is positive definite, as for any vector $v \in R^k$, $v\neq \bm{0}$,
\[
	v^T L^+_{B} v = v^T P_B^T L^+ P_B v =
	\left (v \, \, \, \bm{0} \right )
	L^+
	\begin{pmatrix}
	v \\
	\bm{0}\\
	\end{pmatrix} >0,
\]
where the last inequality follows from the fact that the vector $\left (v \, \, \, \bm{0} \right )$ is not a multiple of the vector $\bm{1}$ and thus does not lie in the null space of $L^+$. The optimization problem in~\eqref{eq:optk_rewritten} has then a unique solution, since the corresponding Hessian is positive definite.

Let $\g \in \R$ be the Lagrange multiplier $\g$ associated with the unique equality constraint of the optimization problem~\eqref{eq:optk_rewritten}. The associated Lagrangian associated is
\[
	\mathcal{L}(\widetilde{\a} ,\g) = \frac{\sigma^2}{2} \widetilde{\a}^T (P_B^T L^+ P_B) \widetilde{\a} - ( P_B^T L^+ \bS \bm{1})^T \widetilde{\a} - \g (\bm{1}^T \widetilde{\a} -1).
\]
Setting $\bm{b}:=P_B^T L^+ \bS \bm{1} \in \R^{k}$, the optimality conditions read
\[
	\begin{cases}
		\sigma^2 L^+_{B} \widetilde{\a} - \g \bm{1} =  \bm{b},\\
		\bm{1}^T \widetilde{\a} =1,
	\end{cases}
\]
or, equivalently, in matrix form
\begin{equation}
	\begin{pmatrix}
	\displaystyle \sigma^2 L^+_{B} & -\bm{1}\\
	\displaystyle \bm{1}^T & 0 \\
	\end{pmatrix}
	\begin{pmatrix}
	\displaystyle \widetilde{\a}\\
	\displaystyle \g  \\
	\end{pmatrix}
	=
	\begin{pmatrix}
	\displaystyle \bm{b}\\
	\displaystyle 1\\
	\end{pmatrix}.
	\label{eq:matrixform}
\end{equation}
Being positive definite, $L^+_{B}$ is invertible and its inverse is also positive definite, which means that $t_B:=\bm{1}^T (L^+_{B})^{-1} \bm{1} >0$. In view of the fact that $\sigma^2 t_B^{-1} \neq 0$ and $L^+_{B}$ is invertible, we can use the block matrix inversion formula given in Lemma~\ref{lem:block} to obtain
\[
	\begin{pmatrix}
	\displaystyle \sigma^2 L^+_{B} & -\bm{1}\\
	\displaystyle \bm{1}^T & 0 \\
	\end{pmatrix}^{-1}
	=
	\begin{pmatrix}
		\displaystyle \frac{1}{\sigma^2} (L^+_{B})^{-1} \Big ( I-\frac{1}{t_B} \bJ (L^+_{B})^{-1} \Big) & \displaystyle \frac{1}{t_B} (L^+_{B})^{-1} \bm{1}\\
		\displaystyle \frac{1}{t_B} \bm{1}^T (L^+_{B})^{-1} & \displaystyle \frac{\sigma^2}{t_B} \\
	\end{pmatrix}.
\]
The solution of the linear system~\eqref{eq:matrixform} then reads
\[
	\begin{pmatrix}
		\displaystyle \widetilde{\a}\\
		\displaystyle \g  \\
	\end{pmatrix}
	=
	\begin{pmatrix}
		\displaystyle \sigma^2 L^+_{B} & -\bm{1}\\
		\displaystyle \bm{1}^T & 0 \\
	\end{pmatrix}^{-1}
	\begin{pmatrix}
		\displaystyle \bm{b}\\
		\displaystyle 1\\
	\end{pmatrix}
	=
	\begin{pmatrix}
		\displaystyle \frac{1}{\sigma^2} (L^+_{B})^{-1} \Big ( I-\frac{1}{t_B} \bJ (L^+_{B})^{-1} \Big) & \displaystyle \frac{1}{t_B} (L^+_{B})^{-1} \bm{1}\\
		\displaystyle \frac{1}{t_B} \bm{1}^T (L^+_{B})^{-1} & \displaystyle \frac{\sigma^2}{t_B} \\
	\end{pmatrix}
	\begin{pmatrix}
		\displaystyle \bm{b}\\
		\displaystyle 1\\
	\end{pmatrix}
\]
and thus the optimal load-sharing vector $\widetilde{\a} \in \R^k$ is given by
\begin{align*}
	\widetilde{\a}^*
	&= \frac{(L^+_{B})^{-1} }{\sigma^2} \left ( I - \frac{1}{t_B} \bJ (L^+_{B})^{-1} \right ) \bm{b} + \frac{1}{t_B} (L^+_{B})^{-1} \bm{1} \\
	& =\frac{1}{t_B} (L^+_{B})^{-1} \bm{1} + (L^+_{B})^{-1}  \left ( I - \frac{1}{t_B} \bJ (L^+_{B})^{-1} \right ) \frac{\bm{b}}{\sigma^2} \\
	&= \frac{1}{t_B} (L^+_{B})^{-1} \bm{1} + \left ( (L^+_{B})^{-1} - \frac{1}{t_B} (L^+_{B})^{-1} \bJ (L^+_{B})^{-1} \right ) P_B^T L^+ \frac{\bS \bm{1}}{\sigma^2}\\
	&= \frac{1}{t_B} (L^+_{B})^{-1} \bm{1} + \left ( I - \frac{1}{t_B} (L^+_{B})^{-1} \bJ \right ) (L^+_{B})^{-1} P_B^T L^+ \frac{\bS \bm{1}}{\sigma^2}.
\end{align*}
We now focus on the special case where $S \subseteq B$ and prove identity~\eqref{eq:speccase}. Rewrite $L^+$ as a block matrix
\[
	\begin{pmatrix}
		\displaystyle L^+_{B} & L^+_{C}\\
		\displaystyle (L^+_{C})^T & L^+_{B^c} \\
	\end{pmatrix},
\]
with $L^+_{B} \in \R^{k \times k}$, $L^+_{C} \in \R^{k \times n-k}$ and $L^+_{B^c} \in \R^{n-k \times n-k}$. Note that $L^+_{B}$ and $L^+_{B^c}$ are symmetric matrices, since $L^+$ is. This is consistent with the former definition of $L^+_B$, since
\[
	P_B^T L^+ P_B = ( I_k | \, \mathbb{O} ) L_+ ( I_k | \, \mathbb{O} )^T = ( I_k | \, \mathbb{O} )^T \begin{pmatrix}
		\displaystyle L^+_{B} & L^+_{C}\\
		\displaystyle (L^+_{C})^T & L^+_{B^c} \\
	\end{pmatrix} ( I_k | \, \mathbb{O} )= L^+_{B}.
\]
We start by noticing that
\[
	(L^+_{B})^{-1} P_B^T L^+ = (L^+_{B})^{-1} ( I_k |\, \mathbb{O} )
	\begin{pmatrix}
		\displaystyle L^+_{B} & L^+_{C}\\
		\displaystyle (L^+_{C})^T & L^+_{B^c} \\
	\end{pmatrix} =
	(L^+_{B})^{-1} ( L^+_{B} |\, L^+_{C} ) = ( I_k |\, (L^+_{B})^{-1} L^+_{C} )  \in \R^{k \times n}.
\]
From the assumption $S \subseteq B$ it follows that the covariance matrix can be rewritten as
\[
	\bS  = \begin{pmatrix}
	\displaystyle \bS_B & \mathbb{O}\\
	\displaystyle \mathbb{O} & \mathbb{O} \\
	\end{pmatrix},
\]
where $\bS_B \in \R^{k \times k}$ is itself a covariance matrix (hence, a symmetric positive semi-definite matrix). Trivially, $ \bmo^T  (  \bS_B |\, \mathbb{O} )\bmo= \tr(\bS_B \bJ) = \tr(\bS \bJ) = \sigma^2$. Furthermore,
\begin{align*}
	(L^+_{B})^{-1} P_B^T L^+ \bS & =  (L^+_{B})^{-1} P_B^T L^+
	\begin{pmatrix}
		\displaystyle \bS_B & \mathbb{O}\\
		\displaystyle \mathbb{O}& \mathbb{O} \\
	\end{pmatrix}
	= ( I_k |\, (L^+_{B})^{-1} L^+_{C} )
	\begin{pmatrix}
		\displaystyle \bS_B & \mathbb{O}\\
		\displaystyle \mathbb{O}& \mathbb{O} \\
	\end{pmatrix}
	= (  \bS_B |\, \mathbb{O} ).
\end{align*}
The optimal control $\widetilde{\a}^*$ then rewrites as
\begin{align*}
	\widetilde{\a}^* &= \frac{(L^+_{B})^{-1}}{t_B} \bmo +  \left (I_k - \frac{(L^+_{B})^{-1}  \bJ}{t_B } \right ) (L^+_{B})^{-1}  P_B^T L^+\frac{\bS \bmo}{ \sigma^2}\\
	&= \frac{(L^+_{B})^{-1}}{t_B} \bmo +  \left (I_k - \frac{(L^+_{B})^{-1}  \bJ}{t_B } \right )  (  \bS_B | \, \mathbb{O} ) \frac{\bmo}{ \sigma^2}\\
	&= \frac{(L^+_{B})^{-1}}{t_B} \bmo +  (  \bS_B | \, \mathbb{O} ) \frac{\bmo}{ \sigma^2} - \frac{(L^+_{B})^{-1}}{t_B }  \bmo  \frac{\bmo^T  (  \bS_B | \, \mathbb{O} ) \bmo}{ \sigma^2}\\
	&= \frac{(L^+_{B})^{-1}}{t_B} \bmo + \frac{\bS \bmo}{ \sigma^2} - \frac{(L^+_{B})^{-1}}{t_B } \bmo\\
	&= \frac{\bS \bmo}{ \sigma^2}. \hfill \qed
\end{align*}

\section{Proof of Theorem~\ref{thm:optn}}
We already rewrote in~\eqref{eq:quadform} the objective function of the optimization problem~\eqref{eq:opt} as a quadratic form in $\a$, 
\[
	\E \cH_s(\a) = \frac{\sigma^2}{2} \a^T A \a -  \bm{b}^T \a  + c,
\]
where $A= L^+$, $\bm{b}= L^+ \bS \bm{1} \in \R^n$, and $c = \tr(\bS L^+)/2 \in \R^+$. Note that, for the purpose of solving the optimization problem~\eqref{eq:opt}, we can ignore the constant term $c$.

Denote by $ 0 = \l_1 < \l_2 < \dots < \l_{n}$ the eigenvalues of the weighted Laplacian matrix $L$ and let $\bm{v}_1,\bm{v}_2,\dots,\bm{v}_n$ be the corresponding orthonormal basis of eigenvectors.  
Consider the representation of the vector $\a \in \R^n$ in this basis, namely
\begin{equation}
\label{eq:representation}
	\a = a_1 \bm{v}_1 + a_2 \bm{v}_2 + \dots + a_n \bm{v}_n,
\end{equation}
with $a_1,\dots,a_n \in \R$. In view of the fact that the rows of $L$ sum up to zero, it immediately follows that $\bm{v}_1=\frac{1}{\sqrt{n}} \bm{1}$. From the constraint~\eqref{eq:asum}, i.e., $\bm{1}^T \a = 1$, it immediately follows that $a_1=1  / \sqrt{n}$. Indeed,
\[
	a_1 \cdot \sqrt{n} = a_1 \cdot \frac{\bm{1}^T \bm{1}}{\sqrt{n}} =  a_1 \bm{1}^T \bm{v}_1 = a_1 \bm{1}^T \bm{v}_1 + a_2 \bm{v}_1^T \bm{v}_2 + \dots + a_n \bm{v}_1^T \bm{v}_n  = \bm{1}^T ( a_1 \bm{v}_1 + a_2 \bm{v}_2 + \dots + a_n \bm{v}_n ) = \bm{1}^T \a = 1.
\]
Defining the real coefficients $\kappa_2, \dots, \kappa_n$ as
\[
	\kappa_i:=\langle \bS \bm{1}, \bm{v}_i \rangle = \bm{1}^T \bS \bm{v}_i, \qquad i=2,\dots,n,
\]
and using the representation~\eqref{eq:representation}, the two terms of the quadratic form above rewrites as
\[
	\a^T L^+ \a = (a_1 \bm{v}_1 + a_2 \bm{v}_2 + \dots + a_n \bm{v}_n)^T L^+ (a_1 \bm{v}_1 + a_2 \bm{v}_2 + \dots + a_n \bm{v}_n) = a_1^2 \bm{v}_1^T L^+ \bm{v}_1 + \sum_{i=2}^n a_i^2 \bm{v}_i^T L^+ \bm{v}_i = \sum_{i=2}^n \frac{a_i^2}{\l_i},
\]
and
\[
	\bm{b}^T \a = \bm{1}^T \bS L^+ \a = \bm{1}^T \bS L^+  (a_1 \bm{v}_1 + a_2 \bm{v}_2 + \dots + a_n \bm{v}_n) = \sum_{i=2}^n \frac{a_i}{\l_i} \bm{1}^T \bS \bm{v}_i =  \sum_{i=2}^n a_i \frac{\kappa_i}{\l_i},
\]
where we used twice the fact that $L^+ \bm{v}_1 = \bm{0}$. The objective function thus can be rewritten as
\begin{equation}
\label{eq:EHabeta}
	\E \cH(\a)
	= \frac{\sigma^2}{2} \a^T A \a -  \bm{b}^T \a
	= \frac{\sigma^2}{2} \sum_{i=2}^n \frac{a_i^2}{\l_i}  - \sum_{i=2}^n a_i \frac{\kappa_i}{\l_i} =: g(a_2,\dots, a_n).
\end{equation}
The optimization problem~\eqref{eq:opt} is therefore equivalent to a unconstrained minimization problem in $n-1$ variables, $a_2,\dots, a_n$, which can be expressed using the newly introduced function $g: \R^{n-1} \to \R$. The gradient of $g$ can be calculated as
$
	\nabla g (a_2,\dots, a_n)=\Big  ( \frac{\sigma^2}{\l_i}  a_i - \frac{\kappa_i}{\l_i}\Big  )_{i=2,\dots,n}
$
and the Hessian is the diagonal matrix $\bm{H}(g) = \sigma^2 \cdot \mathrm{diag}(\l_2^{-1},\dots, \l_n^{-1})$. The Hessian $\bm{H}(g)$ is constant as it does not depend on $a_2,\dots, a_n$ and is a positive definite matrix, since all its diagonal terms are positive, in view of the fact that $\l_i>0$ for $i=2,\dots,n$ and that $\sigma^2 >0$. The problem is then strictly convex and any stationary point satisfying $\nabla g (a_2,\dots, a_n) = \bm{0}$ would then be a minimum for the function $g$. Solving the optimality condition $\nabla g (a_2,\dots, a_n)= \bm{0}$ yields
\[
	a^*_i =\frac{\kappa_i }{\sigma^2}, \qquad i=2,\dots,n.
\]
Consequently, the optimal load-sharing factor vector $\a^*$ is unique and is given by
\[
	\a^* = \frac{1}{\sqrt{n}} \bm{v}_1 + \frac{1}{\sigma^2} \sum_{i=2}^n \kappa_i \bm{v}_i = \frac{1}{n} \bm{1} + \frac{1}{\sigma^2} \sum_{i=2}^n \kappa_i \bm{v}_i.
\]
Lastly, setting $\bm{v}_1 :=\frac{1}{\sqrt{n}} \bm{1}$ and $\kappa_1 := \langle \bS \bm{1}, \bm{v}_1 \rangle = \bm{1}^T \bS \bm{v}_1 =\bm{1}^T \bS \frac{1}{\sqrt{n}} \bm{1}$, the vector $\a^*$ rewrites as
\[
	\a^* = \frac{\kappa_1}{\sigma^2} \frac{1}{\sqrt{n}} \bm{1} + \frac{1}{\sigma^2} \sum_{i=2}^n \kappa_i \bm{v}_i = \frac{1}{\sigma^2} \left (\sum_{i=1}^n \kappa_i \bm{v}_i \right )= \frac{1}{\sigma^2} \left (\sum_{i=1}^n \langle \bS \bm{1}, \bm{v}_i \rangle \, \bm{v}_i \right ) =  \frac{\bS \bm{1}}{\sigma^2}. \hfill \qed
\]

\section{Proof of Theorem~\ref{thm:aveHk}}

The starting point of the proof are two identities that leverage the properties of the pseudoinverse $L^+$ of the graph Laplacian. Firstly,
\begin{align}
	\sum_{B \subseteq V \, : \, |B|=k} L^+ \Big (\sum_{i \in B} \bm{e}_i\Big ) &= L^+ \Big (\sum_{B \subseteq V \, : \, |B|=k}  \sum_{i \in B} \bm{e}_i\Big ) = L^+ \Big (\sum_{i=1}^n \sum_{B \subseteq V \, : \, |B|=k}   \bm{e}_i \mathds{1}_{\{i \in B\}} \Big ) \nonumber \\
	& = L^+ \Big (\sum_{i=1}^n \bm{e}_i \sum_{B \subseteq V \, : \, |B|=k}   \mathds{1}_{\{i \in B\}} \Big ) = L^+ \Big (\sum_{i=1}^n \bm{e}_i  \binom{n-1}{k-1} \Big )\nonumber  \\
	&= \binom{n-1}{k-1} L^+ \Big (\sum_{i=1}^n \bm{e}_i   \Big ) = \binom{n-1}{k-1} L^+ \bm{1} = \bm{0}, \label{eq:linearid}
\end{align}
where we use the fact that in a graph with $n$ nodes, each nodes belong to exactly $\binom{n-1}{k-1}$ subsets of $k$ nodes. We further claim that
\begin{equation}
\label{eq:quadid}
	\sum_{B \subseteq V \, : \, |B|=k} \Big (\sum_{i \in B} \bm{e}_i\Big )^T L^+ \Big (\sum_{i \in B} \bm{e}_i\Big ) = \binom{n-2}{k-1} \tr (L^+),
\end{equation}
with the convention that $\binom{n-2}{n-1}=0$. Since
\[
	\Big (\sum_{i \in B} \bm{e}_i\Big )^T L^+ \Big (\sum_{i \in B} \bm{e}_i\Big ) = \sum_{i \in B} \bm{e}_i^T L^+ \bm{e}_i + \sum_{i,j \in B, \, i \neq j} \bm{e}_i^T L^+ \bm{e}_j,
\]
we can rewrite the LHS of~\eqref{eq:quadid} as
\begin{equation}
\label{eq:intermediate}
	\sum_{B \subseteq V \, : \, |B|=k} \Big (\sum_{i \in B} \bm{e}_i\Big )^T L^+ \Big (\sum_{i \in B} \bm{e}_i\Big ) = \sum_{B \subseteq V \, : \, |B|=k} \Big ( \sum_{i \in B} \bm{e}_i^T L^+ \bm{e}_i \Big ) + \sum_{B \subseteq V \, : \, |B|=k} \Big (\sum_{i,j \in B, \, i \neq j} \bm{e}_i^T L^+ \bm{e}_j\Big ).
\end{equation}
The first term on the RHS of~\eqref{eq:intermediate} can be rewritten as
\begin{align*}
	\sum_{B \subseteq V \, : \, |B|=k} \Big ( \sum_{i \in B} \bm{e}_i^T L^+ \bm{e}_i \Big )
	&= \sum_{i=1}^n \bm{e}_i^T L^+ \bm{e}_i \Big ( \sum_{B \subseteq V \, : \, |B|=k} \mathds{1}_{\{i \in B\}} \Big ) \\
	&= \sum_{i=1}^n \binom{n-1}{k-1} \bm{e}_i^T L^+ \bm{e}_i  = \binom{n-1}{k-1} \sum_{i=1}^n \bm{e}_i^T L^+ \bm{e}_i \\
	&= \binom{n-1}{k-1} \sum_{i=1}^n L^+_{i,i} =\binom{n-1}{k-1} \tr (L^+),
\end{align*}
while the second term on the RHS of~\eqref{eq:intermediate} is equal to
\begin{align*}
	\sum_{B \subseteq V \, : \, |B|=k} \Big (\sum_{i,j \in B, \, i \neq j} \bm{e}_i^T L^+ \bm{e}_j\Big )
	&= \sum_{i\neq j} \bm{e}_i^T L^+ \bm{e}_j \Big ( \sum_{B \subseteq V \, : \, |B|=k} \mathds{1}_{\{i \in B, \, j\in B\}} \Big ) = \sum_{i=1}^n \binom{n-2}{k-2} \bm{e}_i^T L^+ \bm{e}_j \\
	& = \binom{n-2}{k-2} \sum_{i\neq j} \bm{e}_i^T L^+ \bm{e}_j = - \binom{n-2}{k-2} \tr (L^+).
\end{align*}
In the last step we used the fact that $\sum_{i\neq j} \bm{e}_i^T L^+ \bm{e}_j = - \tr (L^+)$, which immediately follows from
\begin{align*}
	\sum_{i\neq j} \bm{e}_i^T L^+ \bm{e}_j + \tr (L^+) &= \sum_{i\neq j} \bm{e}_i^T L^+ \bm{e}_j + \sum_{i} L^+_{i,i} = \sum_{i\neq j} \bm{e}_i^T L^+ \bm{e}_j + \sum_{i} \bm{e}_i^T L^+ \bm{e}_i = \Big (\sum_{k=1}^n e_k \Big )^T L^+  \Big (\sum_{k=1}^n e_k \Big )\\
	&= \bm{1}^T L^+ \bm{1} = 0.
\end{align*}
Hence,~\eqref{eq:intermediate} rewrites as
\[
	\sum_{B \subseteq V \, : \, |B|=k} \Big (\sum_{i \in B} \bm{e}_i\Big )^T L^+ \Big (\sum_{i \in B} \bm{e}_i\Big ) = \Big (\binom{n-1}{k-1} - \binom{n-2}{k-2}\Big ) \tr(L^+) = \binom{n-2}{k-1} \tr(L^+),
\]
which concludes the proof of identity~\eqref{eq:quadid}.

Each load-sharing vector $\a \in \Ak$ can be written as $\a = \frac{1}{k} \sum_{i \in B} \bm{e}_i,$ for some $B \subseteq V$, $|B|=k$. By Proposition~\ref{prop:Eha}, the expected total loss due to stochastic fluctuations when using this load-sharing vector is given by
\[
	\E \cH_s(\a) = \frac{\sigma^2 }{2} \left (\frac{1}{k} \sum_{i \in B} \bm{e}_i\right )^T L^+ \left (\frac{1}{k} \sum_{i \in B} \bm{e}_i\right ) -  \bm{1}^T \bS L^+ \left (\frac{1}{k} \sum_{i \in B} \bm{e}_i\right )  + \frac{1}{2} \tr(\bS L^+).
\]
Therefore,
\[
	\cH_k = \frac{1}{2} \tr(\bS L^+) +  \frac{1}{|\Ak|} \sum_{B \subseteq V \, : \, |B|=k}\left [ \frac{\sigma^2 }{2} \Big (\frac{1}{k} \sum_{i \in B} \bm{e}_i\Big )^T L^+ \Big (\frac{1}{k} \sum_{i \in B} \bm{e}_i\Big ) -  \bm{1}^T \bS L^+ \Big (\frac{1}{k} \sum_{i \in B} \bm{e}_i\Big )\right ].
\]
Using~\eqref{eq:linearid} and~\eqref{eq:quadid}, we get
\begin{align*}
	\cH_k &=  \frac{1}{2} \tr(\bS L^+) +  \sigma^2 \frac{\tr(L^+)}{2} \frac{\binom{n-2}{k-1}}{k^2 \binom{n}{k}} =  \frac{1}{2} \tr(\bS L^+)  + \sigma^2   \frac{\tr(L^+)}{2} \frac{n-k}{k \cdot n (n-1)} \\
	&= \frac{1}{2} \tr(\bS L^+)  +  \sigma^2  \frac{\tr(L^+)}{2 n (n-1)} \frac{1-\frac{k}{n}}{ \frac{k}{n}}=\frac{1}{2} \tr(\bS L^+) + \sigma^2   \frac{\tr(L^+)}{2 (n-1)} \Big(\frac{1}{k} - \frac{1}{n} \Big). \hfill \qed
\end{align*}
\end{document}